\documentclass[10pt]{scrartcl}
\title{Convergence Rates for Inverse Problems with Impulsive Noise}
\author{Thorsten Hohage\footnotemark[2]\ \footnotemark[4] and Frank Werner\footnotemark[3]\ \footnotemark[4]}

\usepackage[english] {babel}
\usepackage[utf8]{inputenc}
\usepackage{amssymb}
\usepackage{amsxtra}
\usepackage{amsthm}
\usepackage{a4wide}
\usepackage{mathrsfs}
\usepackage{exscale}
\usepackage{mathrsfs}
\usepackage{subfigure}
\usepackage{cite}
\usepackage{booktabs}
\usepackage{pgfplots}
\usepackage{tikz}
\pgfplotsset{compat=newest}
\pgfplotsset{plot coordinates/math parser=false}

\newlength\fheight \newlength\fwidth
\DeclareMathOperator*{\argmin}{argmin}
\DeclareMathOperator*{\argmax}{argmax}
\DeclareMathOperator*{\esssup}{ess\,sup}
\newcommand{\X}{\mathcal X}
\newcommand{\Y}{\mathcal Y}
\newcommand{\sol}{f}
\newcommand{\udag}{\sol^\dagger}
\newcommand{\data}{g}
\newcommand{\gdag}{\data^\dagger}
\newcommand{\gobs}{\data^{\rm obs}}
\newcommand{\ualdel}{\widehat{\sol}_\alpha}
\newcommand{\paldel}{\widehat{p}_\alpha}
\newcommand{\manifold}{\mathbb{M}}
\newcommand{\mc}{\mathbb{P}}
\newcommand{\mi}{\manifold\setminus\mc}
\newcommand{\M}{\mathfrak{M}}
\newcommand{\R}{\mathcal{R}}

\newcommand{\Rset}{\mathbb{R}}
\newcommand{\Borel}{\mathfrak{B}}
\newcommand{\dom}{\mathrm{dom}}

\newcommand{\breg}[1]{\mathcal D\left(#1,\udag\right)}

\newcommand{\Cerr}{C_{\rm err}}
\newcommand{\Cbreg}{C_{\rm bd}}
\newcommand{\err}{\mathbf{err}}
\newcommand{\errbound}{\overline{\err}}
\newcommand{\braces}[1]{\left\{ #1\right\}}
\newcommand{\paren}[1]{\left( #1\right)}
\newcommand{\norm}[2]{\left\Vert#1\right\Vert_{#2}}

\newcommand{\abs}[1]{\left|#1\right|}
\newcommand{\meas}[1]{\left|#1\right|}
\newcommand{\cone}{\mathcal{C}}
\begin{document}

\newtheorem{thm}{Theorem}[section]
\newtheorem{prop}[thm]{Proposition}
\newtheorem{lem}[thm]{Lemma}
\newtheorem{cor}[thm]{Corollary}
\newtheorem{rem}[thm]{Remark}
\newtheorem{ex}[thm]{Example}
\newtheorem{ass}{Assumption}
\setlength{\parindent}{0cm}
\maketitle

\renewcommand{\thefootnote}{\fnsymbol{footnote}}

\footnotetext[2]{hohage@math.uni-goettingen.de, +49 (0)551 39 4509}
\footnotetext[3]{f.werner@math.uni-goettingen.de, +49 (0)551 39 12466}
\footnotetext[4]{Institute for Numerical and Applied Mathematics, University of G\"ottingen, Germany}

\renewcommand{\thefootnote}{\arabic{footnote}}
\begin{abstract}
We study inverse problems $F\left(\sol\right) =\data$ 
with perturbed right hand side $\gobs$ corrupted by so-called impulsive 
noise, i.e.\ noise which is concentrated on a small subset of the domain of 
definition of $\data$. It is well known that Tikhonov-type 
regularization with an $L^1$ data fidelity term yields significantly more 
accurate results than Tikhonov regularization with classical $L^2$ data 
fidelity terms for this type of noise.  
The purpose of this paper is to provide a convergence analysis explaining this 
remarkable difference in accuracy. Our error estimates significantly 
improve previous error estimates for Tikhonov regularization with 
$L^1$-fidelity term in the case of impulsive noise. 
We present numerical results which are in good agreement with the predictions 
of our analysis. 
\end{abstract}
\pagestyle{myheadings}
\thispagestyle{plain}
\markboth{T.~Hohage and F.~Werner}{Convergence Rates for Inverse Problems with Impulsive Noise}

\section{Introduction}
A noise vector or noise function $\xi:\manifold\to\Rset$
is called impulsive if $|\xi|$ is large on a small part of its domain 
of definition $\manifold$  and small or zero elsewhere. 
In the latter case the noise vector will be sparse in a discrete setting. 
Impulsive noise occurs in many applications, e.g.\ 
switching noise in powerline communication systems, 
physical measurements with 
malfunctioning receivers or digital image acquisition with faulty memory 
locations. 

In this paper we study such noise models in the context of inverse 
problems described by a forward operator 
$F:D\left(F\right) \subset \X \to \Y$ between Banach spaces $\X$ and 
$\Y$. Most of this paper deals with the case that 
$\Y = L^1(\manifold)$ for some 
open subset $\manifold\subset \Rset^d$. 
$\udag\in D\left(F\right)$ will denote the exact solution, 
and observed data are described by
\begin{equation}\label{eq:opeq}
\gobs = F\left(\udag\right) + \xi .
\end{equation}
A standard method to construct a stable approximation to $\udag$ in 
this setting is to compute a minimizer of a generalized Tikhonov 
functional
\begin{equation}\label{eq:tik_gen}
\ualdel \in\argmin\limits_{\sol \in D\left(F\right)} \left[\frac{1}{\alpha r}\norm{F\left(\sol\right) - \gobs}{L^r(\manifold)}^r
+ \R\left(\sol\right)\right].
\end{equation}
Here $\R:\X\to (-\infty,\infty]$ is a  convex, lower-semicontinuous and proper
penalty functional (e.g.\ $\R(\sol) =\frac{1}{q}\norm{\sol-\sol_0}{\X}^q$  
with $q\geq 1$ and $\sol_0\in \X$), and $\alpha>0$ is a regularization 
parameter. An interesting special case corresponding to denoising problems 
is that $F$ is an embedding operator of a space $\X$ of higher regularity 
into $L^r(\manifold)$. 
It has been observed by many authors that the choice $r=1$ yields 
much better results than $r=2$ in the case of impulsive noise, and 
several algorithms have been proposed to minimize the Tikhonov 
functional for $r=1$, see e.g.\   
\cite{afhl12,cjk10b,cj12,j11,js12,kkm05,lsds11,n02,n04,ygo07,yzy09}.

We will develop a convergence analysis 
explaining this remarkable difference between Tikhonov regularization
with $r=1$ and $r=2$ for impulsive noise. 
Over the last years several general convergence results for generalized 
Tikhonov regularization as $\norm{\xi}{L^r\left(\manifold\right)}\to0$ have been derived covering 
\eqref{eq:tik_gen} both with $r=1$ and $r=2$
(see \cite{bo04,r05,fh10,g10b,cj12,cjk10b}). 
It follows from our analysis (see eq.~\eqref{eq:improvement} and 
Table~\ref{tab:rates}) that these error bounds tend to be 
highly suboptimal for impulsive noise, even though they are likely to be 
order optimal in a supremum over all $\|\xi\|_{L^1\left(\manifold\right)}\leq \delta$. 

We describe the ``strength'' of an impulsive noise vector $\xi$ by 
two nonnegative parameters $\varepsilon$ and $\eta$, and our main 
result will be an error estimate in terms of these parameters. 
We assume that 
\begin{equation}\label{eq:noise_model_1}
\exists~\mc\in\Borel(\manifold): \qquad 
\norm{\xi}{L^1\left(\mi\right)} \leq \varepsilon, \qquad \left|\mc\right| \leq \eta 
\end{equation}
where $\Borel(\manifold)$ denotes the Borel $\sigma$-algebra of $\manifold$. 
This means that the data may be arbitrarily strongly corrupted on a small part $\mc\subset\manifold$ whereas 
the $L^1$--error is small in the remaining part of $\manifold$. 
\eqref{eq:noise_model_1} is a continuous, deterministic noise model. 
Under commonly used discrete, stochastic impulsive noise models 
such as random-valued impulsive noise (RVIN) and in particular 
salt-and-pepper noise 
(see e.g.\ \cite{chn04,cj12}) it is satisfied with discrete 
$\manifold$, $\varepsilon=0$, and some finite $\eta$ with  high probability. 
Note that we do not impose any bound on $|\xi|$ on the set $\mc$. Therefore, 
\eqref{eq:noise_model_1} with positive $\varepsilon$
is also satisfied with high probability 
for more general stochastic noise models involving heavy tails. 

The structure of this paper is as follows: In the following 
section \ref{sec:tik} we review the convergence analysis of  
generalized Tikhonov regularization \eqref{eq:tik_gen} based on a variational formulation of both the source condition and the noise 
level along the lines of \cite{g10b,wh12}. Section \ref{sec:conv} 
contains an error bound in terms of the parameters 
$\varepsilon$ and $\eta$ in \eqref{eq:noise_model_1}, the smoothing 
properties of the operator, and the smoothness of the exact solution 
in terms of a variational source condition. 

In the following section 
we derive rates of convergence by minimizing the right hand side 
of the error bound of the previous section over its parameters. 
For this end we study properties of the function $\varepsilon_{\xi}(\eta)$,  
the minimal value of $\varepsilon$ in \eqref{eq:noise_model_1} for given 
$\xi$ and $\eta$. 
We end this paper by numerical studies demonstrating the sharpness 
of our error bounds in section \ref{sec:num} and some conclusions.

\section{Generalized Tikhonov regularization}\label{sec:tik}
In this section we will set the stage for the subsequent analysis by reviewing with small modifications some known results on generalized Tikhonov regularization in Banach spaces. 

\subsection{well-posedness of Tikhonov regularization}
We first formulate well-known sufficient 
conditions for well-posedness of Tikhonov regularization.
\begin{ass}\label{ass:wellposed}
Let $\X,\Y$ be Banach space, and let $\tau_{\X}$ and $\tau_{\Y}$ 
denote topologies on $\X$ and $\Y$ which are weaker than the norm 
topologies. Moreover, let $F:D(F)\subset\X\to\Y$ be an operator, let
$\R:\X\to (-\infty,\infty]$ a convex, lower semicontinuous functional with  
nonempty 
essential domain $\dom\R := \left\{\sol \in \X ~\big|~ \R\left(\sol\right) < \infty\right\}$ such that 
$\dom\R \subset D(F)$ 
and let $r\in [1,\infty)$. We assume that 
\begin{itemize}
\item all sub-level sets  
$\left\{\sol \in D\left(F\right) ~\big|~ \R\left(\sol\right) 
\leq R\right\}$ for $R\in\Rset$ are sequentially compact 
w.r.t.\ $\tau_{\X}$. 
\item $\|\cdot\|_{\Y}$ is sequentially lower-semicontinuous 
w.r.t.\ $\tau_{\Y}$.
\item $F:\dom\R \to \Y$ is sequentially continuous w.r.t.\
$\tau_{\X}$ and $\tau_{\Y}$. 
\end{itemize}
\end{ass}

Under these conditions the existence of a minimizer 
\begin{equation}\label{eq:tik_gen2}
\ualdel \in\argmin\limits_{\sol \in \X} \left[\frac{1}{\alpha r}\norm{F\left(\sol\right) - \gobs}{\Y}^r
+ \R\left(\sol\right)\right]
\end{equation}
for all $\gobs\in\Y$ can be proven by standard arguments 
(see e.g.\ \cite[Thm.~3.2]{f12} or \cite[Thm.~3.22]{s08} for a proof under 
slightly different assumptions). If $\R$ is strictly convex and $F$ is linear, 
then $\ualdel$ is unique. 
Moreover, under Assumption \ref{ass:wellposed} the minimizers 
$\ualdel$ of \eqref{eq:tik_gen2} are stable w.r.t.\ $\gobs$ 
(see e.g.\ \cite[Thm.~3.3]{f12} or \cite[Thm.~3.23]{s08}).

If $\R\left(\sol\right) =\frac1q\norm{\sol-\sol_0}{\X}^q$, 
$q\geq 1$ for a reflexive Banach space $\X$ and $\sol_0 \in \X$, 
then the assumption on the sublevel sets  
holds true for the weak topology $\tau_{\X}$ on $\X$ and it is natural to consider also the weak topology $\tau_{\Y}$ on $\Y$. Note that
weak sequential continuity of $F$ is a mild assumption which 
holds true in particular for all bounded linear operators. 

\subsection{the data error functional $\err$}
It is instructive to study the case of the  ``most extremely impulsive noise'' where $\xi$ is a sum of $\delta$-peaks:

\begin{ex}\label{ex:delta_peaks}
We choose $\Y$ as the Banach space $\M \left(\manifold\right)$ of all signed finite Borel measures equipped with the total variation norm
\(
\norm{\mu}{\M \left(\manifold\right)} := \left|\mu\right| \left(\manifold\right)
\). 
Recall that $L^1(\manifold)$ can be considered as a closed subspace of $\M(\manifold)$ 
by identifying  a function $g \in L^1\left(\manifold\right)$ with the measure 
$\mu_g\left(A\right) := \int_A g\,\mathrm d x$ and that $\norm{\data}{L^1\left(\manifold\right)} = \|\mu_g\|_{\M}$. 
Let 
\begin{equation}\label{eq:noise_model_2}
\xi = \sum\limits_{j=1}^N c_j \delta_{x_j}
\end{equation}
with $N \in \mathbb N$, $c_j \in \mathbb R$ and $x_j \in \manifold$ for $1 \leq j \leq N$.
Here the classical noise level 
\begin{equation}\label{eq:defi_delta}
\norm{\xi}{\M \left(\manifold\right)} = \sum\limits_{j=1}^N \left|c_j\right|
\end{equation}
might be large. However, for $g \in L^1 \left(\manifold\right)$ we have
\begin{equation}\label{eq:noise_model_2_aux}
\norm{\data-\gobs}{\M \left(\manifold\right)} = \norm{\data-\gdag}{L^1\left(\manifold\right)} + \sum\limits_{j=1}^N \left|c_j\right| = \norm{\data-\gdag}{L^1\left(\manifold\right)} + \norm{\xi}{\M \left(\manifold\right)}.
\end{equation}
This means that the noise $\xi$ influences the data fidelity functional $\data\mapsto 
\norm{\data-\gobs}{\M\left(\manifold\right)}$ 
only in the form of an additive constant $\norm{\xi}{\M \left(\manifold\right)}$ which 
has no influence on the minimizer of the Tikhonov functional \eqref{eq:tik_gen2}. 

Therefore, we expect to be able
to recover the unknown solution exactly in the limit $\alpha\searrow 0$ even though
the classical noise level may be large. Remarkably, one even obtains exact recovery 
with noisy data for finite $\alpha$ if $\udag$ satisfies a specific source condition (see 
Remark \ref{rem:exact_penalization}). 
\end{ex}

This example shows that the norm of $\xi$ is not always a good measure of  
its influence on the reconstruction error. We have to study the influence of 
$\xi$  on the empirical data fidelity functional 
$\data\mapsto \norm{\data-\gobs}{\Y}^r$ more precisely. As in \cite{hw13} 
we will describe the difference of the empirical data fidelity functional 
and the ideal data fidelity functional $\data\mapsto \norm{\data-\gdag}{\Y}^r$
not only by a number, but by a functional $\err$. Obviously, additive
constants do not matter, so we subtract $\|\gdag-\gobs\|^r=\|\xi\|^r$. 
Moreover, it will be important to allow different multiplicative constants. 
This motivates the following assumption, which has been used in 
\cite[Ass.~1]{wh12} with $\err=const$:
\begin{ass}\label{ass:err}
Let $\udag\in\dom\R$ denote the exact solution, let 
$\gdag := F\left(\udag\right)$, and let $\gobs\in\Y$ be the 
observed data. We assume that there exist $\Cerr > 0$ and a 
noise level function 
$\err : F\left(\dom\R\right) \to \left[0,\infty\right]$ such 
that
\begin{equation}\label{eq:err}
\norm{\data-\gobs}{\Y}^r -  \norm{\xi}{\Y}^r \geq \frac{1}{\Cerr}\norm{\data-\gdag}{\Y}^r -\err\left(\data\right),
\qquad \data\in F(\dom\R). 
\end{equation}
\end{ass}

In the following we will bound the reconstruction error in terms of $\err(F(\ualdel))$.
\begin{rem}\label{rem:ass_err} 
\begin{enumerate}
\item
It follows from the triangle inequality that Assumption \ref{ass:err} is always fulfilled with 
\begin{equation}\label{eq:standard_err_bound}
\Cerr = 2^{r-1}\qquad \mbox{and}\qquad \err \equiv 2\norm{\xi}{\Y}^r.
\end{equation}
\item In Example \ref{ex:delta_peaks} (see eq.~\eqref{eq:noise_model_2_aux}), 
Assumption \ref{ass:err} holds true with the optimal parameters 
\[
\Cerr = 1\qquad  \mbox{and} \qquad \err \equiv 0.
\] 
\end{enumerate}
\end{rem}

\subsection{error bounds}
The derivation of convergence rates for inverse problems always requires some a priori knowledge on the unknown solution $\udag$, as otherwise the rate of convergence may be arbitrarily slow (see \cite{ehn96}). 
Convergence rates are usually measured w.r.t.\ the Bregman distance
\[
\breg{\sol} := \R \left(\sol\right) - \R \left(\udag\right) - \left<\sol^*, \sol-\udag\right>
\]
where $\sol^* \in \partial \R \left(\udag\right)$ is a subgradient. 
Note that $\breg{\sol}$ depends on $\R$ and the choice of $\sol^*$ 
(unless $\partial \R \left(\udag\right)$ is a singleton), but we 
omit this dependence in our notation. 
The Bregman distance has first been used for the convergence analysis 
of generalized Tikhonov regularization by Eggermont \cite{e93} for 
maximum entropy regularization and for more general penalty functionals 
by Burger \& Osher \cite{bo04}. 
For $\R\left(\sol\right) = \frac12\norm{\sol-\sol_0}{\X}^2$ with a Hilbert norm $\norm{\cdot}{\X}$ one obtains $\breg{\sol}= \frac12\norm{\sol-\udag}{\X}^2$. In this sense, the Bregman distance is a natural generalization of the norm. 

Recently, in a number of papers \cite{bh10,hy10,fh11,g10b} rates of convergence of generalized Tikhonov regularization
have been analyzed using a variational formulation of the 
source condition: 
\begin{ass}\label{ass:vie}
Suppose the variational inequality 
\begin{equation}\label{eq:vie}
\beta \breg{\sol} \leq \R\left(\sol\right)-\R\left(\udag\right) + \varphi\left(\norm{F\left(\sol\right)-\gdag}{\Y}^r\right)\qquad \mbox{for all } \sol\in\dom\R
\end{equation}
holds true with some $\beta>0$ and a concave index function $\varphi : \left[0,\infty\right) \to \left[0,\infty\right)$ (i.e.\ $\varphi$ monotonically increasing, $\varphi \left(0\right) = 0$).\footnote{Note that the concavity of $\varphi$ implies that $-\varphi$ is convex and due to finiteness thus also continuous (see \cite[Cpt.~1, Cor.~2.3]{et76}).}
\end{ass}

In a Hilbert space setup with $\R\left(\sol\right) = \left\Vert \sol-\sol_0\right\Vert_{\X}^2$ and a bounded linear operator $F = T$ it has been shown that Assumption \ref{ass:vie} is in general weaker than spectral source conditions yielding the same convergence rates \cite{fhm11}, and also for general $\R$ and Fr\'echet-differentiable $F$ having a Lipschitz continuous derivative $F'$ w.r.t. the Bregman distance it is known that 
\eqref{eq:vie} with $\varphi \left(t\right) = ct^{1/r}$ is equivalent to the so-called {\it Benchmark source condition}
\begin{equation}\label{eq:sc_bench}
F'\left[\udag\right]^*\omega \in \partial \R \left(\udag\right)
\end{equation}
for some $\omega\in \Y^*$ (see e.g.\ \cite[Prop.~3.35 \& 3.38]{s08}). This also shows that \eqref{eq:vie} can in general be seen as a combination of source and nonlinearity condition.

As first noticed by Grasmair \cite{g10b}  the approximation error can be bounded in terms of 
the Fenchel conjugate  of $-\varphi$, which is defined by
\begin{equation}\label{eq:defi_Fenchel}
\left(-\varphi\right)^* \left(s\right) = \sup\limits_{\tau \geq 0}\left(s\tau + \varphi \left(\tau\right)\right),  \qquad s < 0,
\end{equation}
(see e.g.\ \cite[Cpt.~3]{et76} for more information about Fenchel duality). 
More precisely it will turn out that the approximation error is bounded by
the function $\psi : \left(0,\infty\right) \to \left[0,\infty\right]$, 
\begin{equation}\label{eq:defi_psi}
\psi\left(\alpha\right) := \left(-\varphi\right)^*\left(-\frac{1}{\alpha}\right), \qquad \alpha > 0.
\end{equation}
Now we are ready to prove the following convergence estimates which extend 
\cite[Thm.~3.3]{wh12} by error bounds in $\Y$. Such error bounds are interesting 
in particular for denoising problems (i.e.\ $F$ an embedding operator). Error bounds in $\Y$ under 
variational source conditions are also known (see  \cite{hm12}), but only under the classical noise 
level $\err \equiv 2 \norm{\xi}{\Y}^r$ (cf. Remark \ref{rem:ass_err}).
\begin{thm}\label{thm:conv}
Under Assumptions \ref{ass:wellposed}, 
\ref{ass:err} and \ref{ass:vie} the following holds true:
\begin{enumerate}
\item {\bf Bounds for the minimizers:} 
\begin{equation}\label{eq:min_bounds}
\R\left(\ualdel\right) \leq  \frac{\err\left(F\left(\ualdel\right)\right) }{r\alpha} + \R\left(\udag\right) 
\end{equation}
for all $\alpha > 0$ and all minimizers $\ualdel$ in 
\eqref{eq:tik_gen2}.
\item {\bf Error decomposition:} For all $\alpha > 0$ and $\lambda \in \left(0,1\right)$ we have 
\begin{subequations}\label{eqs:error_decomposition}
\begin{align}\label{eq:error_decomp_breg}
\beta\breg{\ualdel} &\leq \frac{\err \left(F\left(\ualdel\right)\right)}{r\alpha} +\psi\left(r\Cerr\alpha\right) \\[0.1cm]
\label{eq:error_decomp_residuals}
\norm{F\left(\ualdel\right)-\gdag}{\Y}^r &\leq \frac{\Cerr}{\lambda}\err \left(F\left(\ualdel\right)\right) + \frac{r\Cerr\alpha}{\lambda}\psi\left(\frac{r\Cerr\alpha}{1-\lambda}\right).
\end{align}
\end{subequations}
\item {\bf Convergence rates:} If $\Cerr$ is chosen such that $\errbound := \sup_{\sol \in \dom\R} \err \left(F\left(\sol\right)\right)$ is finite (see Remark \ref{rem:ass_err}), 
then the infimum of the right-hand side of \eqref{eq:error_decomp_breg} with $\err \left(F\left(\ualdel\right)\right)$ replaced by $\errbound$ is attained if and only if $\alpha$ is chosen a priori such that 
\begin{equation}\label{eq:choice_alpha}
\frac{-1}{r\Cerr\alpha} \in \partial\left(-\varphi\right)\left(r\Cerr\errbound\right).
\end{equation}
For $\alpha$ as in \eqref{eq:choice_alpha} we have
\begin{subequations}
\begin{equation}\label{eq:rates_breg}
\breg{\ualdel} = \mathcal O \left(\varphi \left(\errbound\right)\right) \qquad\text{as}\qquad \errbound \searrow 0.
\end{equation}
If moreover $\varphi \left(t\right) = c \cdot t^\kappa$ with $\kappa \in \left(0,1\right]$ and $c > 0$, then
\begin{equation}\label{eq:rates_residuals}
\norm{F\left(\ualdel\right)-\gdag}{\Y}^r = \mathcal O \left(\errbound\right)\qquad\text{as}\qquad \errbound \searrow 0.
\end{equation}
\end{subequations}
\end{enumerate}
\end{thm}
\begin{proof}
\begin{enumerate}
\item By the definition of $\ualdel$ we have after multiplication by $r\alpha$ that
\begin{equation}\label{eq:aux_conv}
\norm{F\left(\ualdel\right)-\gobs}{\Y}^r + r\alpha \R(\ualdel) \leq \norm{F\left(\udag\right)-\gobs}{\Y}^r + r\alpha \R(\udag) = \norm{\xi}{\Y}^r + r\alpha \R(\udag).
\end{equation}
Inserting \eqref{eq:err} and dividing by $r\alpha$ implies
\begin{equation}\label{eq:aux_conv_2}
\R\left(\ualdel\right) \leq  - \frac{1}{r\Cerr\alpha}\norm{F\left(\ualdel\right)-\gdag}{\Y}^r + \frac{\err\left(F\left(\ualdel\right)\right) }{r\alpha} + \R\left(\udag\right),
\end{equation}
which especially proves \eqref{eq:min_bounds}.
\item It follows from \eqref{eq:aux_conv_2} that
\begin{align*}
\beta \breg{\ualdel} \stackrel{\eqref{eq:vie}}{\quad\leq\quad}&\R\left(\ualdel\right) - \R\left(\udag\right) + \varphi\left(\norm{F\left(\ualdel\right)-\gdag}{\Y}^r\right)\\[0.1cm]
\stackrel{\eqref{eq:aux_conv_2}}{\quad\leq\quad}& \frac{\err \left(F\left(\ualdel\right)\right)}{r\alpha} - \frac{1}{r\Cerr\alpha}\norm{F\left(\ualdel\right)-\gdag}{\Y}^r + \varphi\left(\norm{F\left(\ualdel\right)-\gdag}{\Y}^r\right)\\[0.1cm]
\quad\leq\quad & \frac{\err \left(F\left(\ualdel\right)\right)}{r\alpha} - \frac{\lambda}{r\Cerr\alpha}\norm{F\left(\ualdel\right)-\gdag}{\Y}^r + \sup\limits_{\tau\geq 0} \left[ \frac{\tau\left(1-\lambda\right)}{-r\Cerr \alpha} - \left(-\varphi\right)\left(\tau\right)\right] \\[0.1cm]
\stackrel{\eqref{eq:defi_Fenchel}}{\quad=\quad}& \frac{\err \left(F\left(\ualdel\right)\right)}{r\alpha} - \frac{\lambda}{r\Cerr\alpha}\norm{F\left(\ualdel\right)-\gdag}{\Y}^r + (-\varphi)^*\left(-\frac{1-\lambda}{r\Cerr\alpha}\right)\\[0.1cm]
\quad=\quad&\frac{\err \left(F\left(\ualdel\right)\right)}{r\alpha} - \frac{\lambda}{r\Cerr\alpha}\norm{F\left(\ualdel\right)-\gdag}{\Y}^r +\psi\left(\frac{r\Cerr\alpha}{1-\lambda}\right)
\end{align*}
for all $\alpha > 0$ and $\lambda \in \left[0,1\right)$. The choice $\lambda = 0$ implies \eqref{eq:error_decomp_breg} and if $\lambda > 0$, then \eqref{eq:error_decomp_residuals} follows by rearranging terms and the non-negativity of the Bregman distance.
\item The assertion on the infimum on the right-hand side of \eqref{eq:error_decomp_breg} and the corresponding convergence rate \eqref{eq:rates_breg} follows from \cite[Thm~3.3, 2.]{wh12}. For $\varphi \left(t\right) = c \cdot t^\kappa$ one 
readily sees that \eqref{eq:choice_alpha} is equivalent to
$\alpha = \frac{1}{c\kappa r^\kappa\Cerr^\kappa} \errbound^{1-\kappa}$ and  
\begin{equation}\label{eq:fenchel_conj_power_funct}
(-\varphi)^*\left(-s\right) = Cs^{\frac{\kappa}{\kappa-1}} , \qquad C = \left(c\left(\kappa c\right)^{-\frac{\kappa}{\kappa-1}} -\left(\kappa c\right)^{-\frac{1}{\kappa-1}}\right).
\end{equation}
Thus the error estimate \eqref{eq:rates_residuals} yields
\begin{align*}
\norm{F\left(\ualdel\right)-\gdag}{\Y}^r&\leq \frac{\Cerr}{\lambda}\errbound + \frac{r\Cerr}{\lambda} \frac{1}{c\kappa r^\kappa\Cerr^\kappa} \errbound^{1-\kappa} C \left(\frac{r\Cerr\frac{1}{c\kappa r^\kappa\Cerr^\kappa} \errbound^{1-\kappa}}{1-\lambda}\right)^{\frac{\kappa}{1-\kappa}} \\[0.1cm]
&= \mathcal O\left(\errbound\right),
\end{align*}
i.e.\ we obtain the expected convergence rate for the residuals.
\end{enumerate}
\end{proof}

\begin{rem}[Benchmark source condition and exact penalization]\label{rem:exact_penalization}
Suppose that $r = 1$, that $F = T$ is bounded and linear and that the benchmark source condition \eqref{eq:sc_bench} holds true. If we choose $\sol^* = T^*\omega \in \partial \R \left(\udag\right)$ for the definition of the Bregman distance, it can readily be seen  from the estimate
\[
\breg{\sol} - \R\left(\sol\right) + \R\left(\udag\right) = \langle \sol^*, \sol-\udag\rangle = \langle \omega, T\sol - \gdag\rangle \leq \norm{\omega}{\Y^*} \norm{T\sol - \gdag}{\Y}
\]
that Assumption \ref{ass:vie} holds true with $\beta = 1$ and $\varphi \left(t\right) = \norm{\omega}{\Y^*}t$. An easy calculation shows that
\[
\left(-\varphi\right)^* \left(s\right) = \begin{cases} 0 & \text{if } s \leq - \norm{\omega}{\Y^*}, \\[0.1cm] \infty & \text{otherwise} \end{cases}
\]
in this case. This implies in particular that for $\err\equiv 0$ we have 
\[
\breg{\ualdel} = 0 \qquad\text{whenever}\qquad \alpha \leq \frac{1}{\Cerr\norm{\omega}{\Y^*}},
\]
which is known as effect of {\it exact penalization} (see e.g.\ \cite[Sec.~3.2]{bo04}). This result can obviously be generalized to the nonlinear case provided \eqref{eq:vie} holds true with $r = 1$, $\varphi \left(t\right) = c \cdot t$ and arbitrary $\beta > 0$.
\end{rem}

\section{Error bounds in terms of $\eta$, $\varepsilon$, and $\alpha$}\label{sec:conv}
In this section we will analyze Tikhonov regularization \eqref{eq:tik_gen} with $r = 1$, i.e.\ 
$\Y = L^1 \left(\manifold\right)$. Most of this section is concerned with the estimation of the data error functional: 
For given $\varepsilon,\eta\geq 0$ we have to specify a function $\err:F(\dom\R)$ and a constant $\Cerr\geq 1$
such that 
\[
\norm{\data-\gobs}{L^1(\manifold)} -  \norm{\xi}{L^1(\manifold)} 
\geq \frac{1}{\Cerr}\norm{\data-\gdag}{L^1(\manifold)} -\err\left(\data\right)\, 
\]
for all $\xi\in L^1(\manifold)$ satisfying \eqref{eq:noise_model_1} and all $\data\in F(\dom\R)$. 
Then error bounds in $\X$ and $\Y$ will follow from eq.~\eqref{eqs:error_decomposition} in Theorem \ref{thm:conv}.

\subsection{estimation of the data error functional $\err$}
First note that for all $\data,\gdag \in L^\infty \left(\manifold\right)$ and  all $\xi\in L^1(\manifold)$ satisfying 
\eqref{eq:noise_model_1} we have
\begin{align}\label{eq:lower_bound_err}
\norm{\data-\gobs}{L^1 \left(\manifold\right)} -\norm{\xi}{L^1 \left(\manifold\right)} 
&= \int\limits_{\mi}\left[\left|\gobs - \data\right| - \left|\gobs - \gdag\right|\right] \,\mathrm d x + \int\limits_{\mc}\left[\left|\gobs - \data\right| - \left|\gobs - \gdag\right|\right] \,\mathrm d x\nonumber\\[0.1cm]
&\geq \norm{\data-\gdag}{L^1\left(\mi\right)} - 2\varepsilon - \left|\mc\right|\norm{g-\gdag}{L^\infty\left(\mc\right)} \\[0.1cm]
& \geq \norm{\data-\gdag}{L^1 \left(\manifold\right)} - 2 \left(\varepsilon +\left|\mc\right| \norm{\data-\gdag}{L^\infty\left(\mc\right)} \right)\nonumber
\end{align}
where we have used the first triangle inequality in the form 
$\left|a-b\right| - \left|a-c\right| \geq \left|c-b\right| - 2 \left|a-c\right|$ on $\mi$ 
and the second triangle inequality on $\mc$. 
To proceed we need to assume that $F$ maps into a Sobolev space 
$W^{k,p}(\manifold)$ with norm 
$\norm{\data}{W^{k,p}(\manifold)}:= \big(\sum_{|\alpha|\leq k}
\norm{D^{\alpha}\data}{L^p(\manifold)}^p\big)^{1/p}$. We also need 
the seminorm 
$|\data|_{W^{k,p}(\manifold)}:=
\big(\sum_{|\alpha|=k}\norm{D^{\alpha}\data}{L^p(\manifold)}^p\big)^{1/p}$. 

\begin{ass}[smoothing properties of the forward operator]\label{ass:T}
$\manifold \subset \mathbb R^d$ is a bounded Lipschitz domain 
and there exist $k \in \mathbb N_0$, $p\in [1,\infty]$ and $q\in(1,\infty)$ such that
\begin{equation}
F(\dom\R )\subset W^{k,p} \left(\manifold\right) \qquad \mbox{and}\qquad 
\abs{F\left(\sol\right) - \gdag}_{W^{k,p} \left(\manifold\right)} \leq C_{F,k,p} \breg{\sol}^{\frac{1}{q}}
\end{equation}
for all $\sol \in \dom\R$ with some $C_{F,k,p} > 0$. 
\end{ass}

Obviously, if $F$ is linear, $\X$ is a Hilbert space, $\R\left(\sol\right) =\frac12 \norm{\sol}{\X}^2$ and $q=2$, then 
Assumption \ref{ass:T} reduces to boundedness of $F$ from $\X$ to 
$W^{k,p}(\manifold)$. The same holds true for many important 
Banach spaces $\X$ and different values of $q$:
If $\X$ is a $q$-convex Banach space, $q > 1$, and 
$\R\left(\sol\right) = \frac1q \norm{\sol}{\X}^q$, then it follows from the 
inequalities of Xu \& Roach (see \cite[eq.~(2.17)']{xr91}) that
\begin{equation}\label{eq:bregman_inequality}
\norm{\sol-\udag}{\X} \leq \Cbreg \breg{\sol}^{\frac1q}
\end{equation}
for all $\sol \in \X$ with some $\Cbreg > 0$. So if $F : \X \to  W^{k,p} \left(\manifold\right)$ is Lipschitz continuous with constant $L_{F,k,p}>0$, $\X$ is $q$-convex
and $\R(\sol)=\frac{1}{2}\norm{\sol}{\X}^2$, then Assumption \ref{ass:T} is fulfilled with $C_{F,k,p} = \Cbreg L_{F,k,p}$. Examples for $2$-convex Banach spaces are $\ell^p$, $L^p \left(\Omega\right)$ and $W^{m,p} \left(\Omega\right)$ for any $1 < p \leq 2$. 
The spaces $\ell^p$, $L^p \left(\Omega\right)$ and $W^{m,p} \left(\Omega\right)$ with $p > 2$ are $p$-convex. Note that inequalities of the form \eqref{eq:bregman_inequality} can hold true also for more general penalty functionals $\R$, e.g.\ the maximum entropy functional (see \cite{bl91}). 

\begin{rem}\label{rem:err_Linf}
If Assumption \ref{ass:T} holds true with $k=0$ and $p=\infty$, then Assumption \ref{ass:err} is fulfilled with
\begin{equation}\label{eq:err_1}
\Cerr = 1\qquad\text{and}\qquad \err\left(F\left(\sol\right)\right) = 2 \varepsilon + 2\eta C_{F,0,\infty} \breg{\sol}^{\frac{1}{q}}
\end{equation}
for all $\xi\in L^1(\manifold)$ satisfying \eqref{eq:noise_model_1}.
\end{rem}

If stronger smoothing properties of $F$ are assumed, the simple estimate in 
Remark~\ref{rem:err_Linf} can be improved. We first need the following lemma:

\begin{lem}\label{lem:normbound}
If Assumption \ref{ass:T} holds true with  $k>d/p$, then  
there exist constants $c_1,c_2$, and $\rho_0>0$ such that 
\begin{equation}\label{eq:sampling}
\norm{\data}{L^{\infty}(\manifold)}\leq c_1 \rho^{k-\frac{d}{p}}
\abs{\data}_{W^{k,p}(\manifold)} 
+\frac{c_2}{\rho^d}\norm{\data}{L^1\left(\manifold\right)}
\end{equation}
for all $\data\in W^{k,p}(\manifold)$ and $\rho\in(0,\rho_0]$. 
\end{lem}

\begin{proof}
For $\rho>0$ and $\theta\in(0,1)$ define a cone with radius $\rho$ and aperture angle $2\cos^{-1}(\theta)$ by
\[
\cone(\rho,\theta):= \{x\in\Rset^d:0<\abs{x}_2<\rho, x_1< \theta \abs{x}_2\}.
\]
By Sobolev's embedding theorem (see e.g.\ \cite[\S 6.4.6]{RR:92}) there exists a constant $C_{\theta}>0$ such that 
\[
\norm{\data}{L^{\infty}(\cone(1,\theta))}\leq C_{\theta}\norm{\data}{W^{k,p}(\cone(1,\theta))}
\]
for all $\data\in W^{k,p}(\cone(1,\theta))$. By an application of Ehrling's lemma 
(see \cite[Thm.~6.99 and Cor.~6.100]{RR:92}) there exist constants $c_1$ and $c_2$ such that 
\[
\norm{\data}{L^{\infty}(\cone(1,\theta))}
\leq c_1\abs{\data}_{W^{k,p}(\cone(1,\theta))} + c_2\norm{\data}{L^1(\cone(1,\theta))}
\]
for all $\data\in W^{k,p}(\cone(1,\theta))$. For $\rho>0$ define $\data_{\rho}(x):=\data(x/\rho)$ 
and note that the mapping $\data\mapsto \data_{\rho}$ is a isomorphism from 
$W^{k,p}(\cone(1,\theta))\to W^{k,p}(\cone(\rho,\theta))$. Moreover, a straightforward computation 
shows that 
\[
\norm{\data_\rho}{L^{\infty}(\cone(\rho,\theta))}
\leq \rho^{k-\frac{d}{p}}c_1\abs{\data_\rho}_{W^{k,p}(\cone(\rho,\theta))} 
+ \frac{c_2}{\rho^d}\norm{\data_\rho}{L^1(\cone(\rho,\theta))}
\]
for all $\data_\rho\in W^{k,p}(\cone(\rho,\theta))$. Since $\manifold$ is Lipschitz it satisfies 
the uniform interior cone property (see \cite[Thm.~2.1]{w87}), 
i.e.\ there exist $\theta\in (0,1)$ and $\rho_0>0$ such that for all $x\in\overline{\manifold}$ there exists a cone 
$\cone_{x,\rho_0}:= x+O_x\cone(\rho_0,\theta)$ with an orthogonal matrix $O_x$ which is contained in $\manifold$. 
Since by Sobolev's embedding theorem every $\data\in W^{k,p}(\manifold)$ is continuous and $\overline{\manifold}$ 
is compact, there exists $y\in\overline{\manifold}$ such that $\abs{\data(y)}=\norm{\data}{L^{\infty}(\manifold)}$. 
Then 
\begin{align*}
\norm{\data}{L^{\infty}(\manifold)} &= \abs{\data(y)} = \norm{\data}{L^\infty(\cone_{y,\rho})}\\
&\leq  \rho^{k-\frac{d}{p}}c_1\abs{\data}_{W^{k,p}(\cone_{y,\rho})} 
+ \frac{c_2}{\rho^d}\norm{\data}{L^1(\cone_{y,\rho})}\\
&\leq  \rho^{k-\frac{d}{p}}c_1\abs{\data}_{W^{k,p}(\manifold)} 
+ \frac{c_2}{\rho^d}\norm{\data}{L^1(\manifold)}
\end{align*}
for all $\rho\leq \rho_0$.
\end{proof}

\begin{prop}\label{prop:kp_errbound}
If Assumption \ref{ass:T} holds true with $k>d/p$, then for all $\Cerr>1$ 
there exist constants $C,\eta_0>0$ such that Assumption \ref{ass:err} is fulfilled with
\begin{equation}\label{eq:err_2}
\err\left(F\left(\sol\right)\right)
= 2\varepsilon + C\eta^{\frac{k}{d}+\frac{p-1}{p}}\breg{\sol}^{\frac{1}{q}}
\end{equation}
for all $\xi$ satisfying \eqref{eq:noise_model_1} with $0 \leq \eta \leq \eta_0$.
\end{prop}

\begin{proof}
From \eqref{eq:lower_bound_err}, Lemma~\ref{lem:normbound} with $\rho^d=\frac{2c_2\Cerr}{\Cerr-1}\eta$ , 
and Assumption \ref{ass:T} we obtain 
\begin{eqnarray*}
\lefteqn{\norm{F\left(\sol\right)-\gobs}{L^1 \left(\manifold\right)} -\norm{\xi}{L^1 \left(\manifold\right)} 
\geq  \norm{F(\sol)-\gdag}{L^1 \left(\manifold\right)} 
- 2 \varepsilon -2 \eta \norm{F(\sol)-\gdag}{L^\infty\left(\mc\right)}} \\
&\geq& \paren{1-2c_2\frac{\eta}{\rho^d}}\norm{F(\sol)-\gdag}{L^1 \left(\manifold\right)} 
-2\varepsilon 
-2\eta c_1 (\rho^d)^{\frac{k}{d}-\frac{1}{p}}\abs{F(\sol)-\gdag}_{W^{k,p}(\manifold)}\\
&\geq& \frac{1}{\Cerr}\norm{F(\sol)-\gdag}{L^1 \left(\manifold\right)} - \err\left(F\left(\sol\right)\right)
\end{eqnarray*}
for all $\sol\in\dom\R$ and $\eta\leq \eta_0$ with $\eta_0:= \frac{\Cerr-1}{2c_2\Cerr} \rho_0^d$ and
$C:=  2c_1 \paren{\frac{2c_2\Cerr}{\Cerr-1}}^{\frac{k}{d}-\frac{1}{p}} C_{F,k,p}$. 
\end{proof}

\subsection{error bound}
In \eqref{eq:err_1} and \eqref{eq:err_2} we have proven for different values of $k$ and $p$  that
\begin{equation}\label{eq:err_leg_breg}
\err \left(F\left(\ualdel\right)\right)  \leq 2\varepsilon + C\eta^{\frac{k}{d}+\frac{p-1}{p}}\breg{\ualdel}^{\frac{1}{q}}
\end{equation}
for all $\eta < \eta_0$ with some constant $C > 0$. Vice versa, in \eqref{eq:error_decomp_breg} we have shown 
an upper bound of $\breg{\ualdel}$ in terms of $\err \left(F\left(\ualdel\right)\right)$. Combining these 
two inequalities we can eliminate $\err \left(F\left(\ualdel\right)\right)$ in the upper bounds for  $\breg{\ualdel}$
and $\norm{F\left(\ualdel\right) - \gdag}{L^1 \left(\manifold\right)}$: 

\begin{thm}\label{thm:impulsive_noise_conv_2}
Suppose Assumptions \ref{ass:wellposed} and \ref{ass:vie} hold true with $\Y = L^1 \left(\manifold\right)$
and $r=1$, the error $\xi$ in \eqref{eq:opeq} fulfills \eqref{eq:noise_model_1} for some $\epsilon,\eta\geq 0$, and 
Assumption \ref{ass:T} holds true either with $k=0$ and $p=\infty$ or with $k>d/p$. 
Let $\ualdel$ be a minimizer of \eqref{eq:tik_gen} and 
$q'\in (1,\infty)$ such that $\frac{1}{q}+\frac{1}{q'}=1$.
Then there exists a constant $C_{\psi} > 0$  such that
\begin{subequations}\label{eqs:error_decomp_final}
\begin{align}
\beta\breg{\ualdel} &\leq 2q'\frac{\varepsilon}{\alpha} + (q'-1)\frac{\eta^{\frac{q'k}{d}
+\frac{q'\left(p-1\right)}{p}}}{\alpha^{q'}} + C_{\psi} \psi\left(\Cerr \alpha\right) 
\label{eq:error_decomp_final} \\[0.1cm]
\norm{F\left(\ualdel\right) - \gdag}{L^1 \left(\manifold\right)} & \leq 4q' \varepsilon + 
2(q'-1)\frac{\eta^{\frac{q'k}{d}+\frac{q'\left(p-1\right)}{p}}}{\alpha^{q'-1}} 
+ 2C_{\psi}\Cerr\alpha \psi\left(2 \Cerr \alpha\right)
\label{eq:error_decomp_final_Y}
\end{align}
\end{subequations}
for all $\alpha > 0$, $\varepsilon>0$ and $0 < \eta < \eta_0$. 
Here we use the convention $\frac{q'\left(p-1\right)}{p} = q'$ for $p = \infty$.
\end{thm}

\begin{proof}
Let $\gamma := q'\paren{\frac{k}{d}+\frac{p-1}{p}}$.
 We insert the error bound 
\eqref{eq:error_decomp_breg} from Theorem \ref{thm:conv} and use the inequality $(a+b)^{\frac{1}{q}}
\leq a^{\frac{1}{q}}+b^{\frac{1}{q}}$ and 
Young's inequality $ab\leq \frac{1}{q'}a^{q'}+\frac{1}{q}b^q$, which both hold for all $a,b\geq 0$, to obtain
\begin{align*}
\err \left(F\left(\ualdel\right)\right)
& \leq 2\varepsilon + C\eta^{\gamma/q'}
\paren{\frac{1}{\alpha\beta}\err\left(F\left(\ualdel\right)\right)+\frac{1}{\beta}\psi\left(\Cerr\alpha\right)}^{1/q}\\
& \leq  2\varepsilon + \frac{C\eta^{\gamma/q'}}{(\alpha\beta)^{1/q}}\err\left(F\left(\ualdel\right)\right)^{1/q}
+ \frac{C\eta^{\gamma/q'}}{(\alpha\beta)^{1/q}}\paren{\alpha\psi\left(\Cerr\alpha\right)}^{1/q}\\
&\leq 2\varepsilon + \frac{1}{q}\err\left(F\left(\ualdel\right)\right) + \frac{\alpha}{q}\psi\left(\Cerr\alpha\right)
+ \frac{2C^{q'}}{q'\beta^{q'/q}}\frac{\eta^{\gamma}}{\alpha^{q'/q}}.
\end{align*}
Subtracting  $\frac{1}{q}\err\left(F\left(\ualdel\right)\right)$ and multiplying by $q'$ on both sides and 
using the identity $\frac{q'}{q}=q'-1$ we obtain 
\[
\err \left(F\left(\ualdel\right)\right) \leq 2q' \varepsilon + (q'-1) \frac{\eta^{\frac{q'k}{d}+\frac{q'\left(p-1\right)}{p}}}{\alpha^{q'-1}} 
+ C_{\psi} \alpha\psi\left(\Cerr\alpha\right)
\]
for all $0 < \eta < \eta_0$ with $\psi$ in \eqref{eq:defi_psi}, $\eta_0$ in Proposition~\ref{prop:kp_errbound}, 
and with $C_{\psi}:= 2C^{q'}\beta^{1-q'}$. 
Plugging this into the error bounds \eqref{eq:error_decomp_breg} and 
\eqref{eq:error_decomp_residuals} (with $\lambda=\frac{1}{2}$) yields \eqref{eqs:error_decomp_final}.
\end{proof}

\section{Convergence rates}\label{sec:discussion}
In this section we prove some rates of convergence based on Theorem \ref{thm:impulsive_noise_conv_2}.
\subsection{convergence rates in terms of 
$\eta$ and $\varepsilon$}
First we derive an explicit order optimal bound an the infimum over $\alpha$ of the right hand sides of 
\eqref{eqs:error_decomp_final} yielding rates of convergence in terms of $\eta$ and $\varepsilon$:

\begin{thm}\label{thm:rate_epsilon_eta}
Suppose the assumptions of Theorem~\ref{thm:impulsive_noise_conv_2} hold true and let  
$\gamma := \frac{q'k}{d}+\frac{q'\left(p-1\right)}{p}$.\\
\emph{Case 1: $\varphi^{1+\delta}$ is concave for some 
$\delta>0$.} If 
\[
\theta \left(\alpha\right) := \alpha \cdot \psi\left(\alpha\right)\qquad\mbox{and}\qquad 
\tilde \theta \left(\alpha\right) := \alpha^{q'} \psi\left(\alpha\right)
\] 
with $\psi$ defined in \eqref{eq:defi_psi} 
and if $\alpha$ is chosen such that 
\[
\underline{c}\alpha 
\leq \theta^{-1}\!\left(\varepsilon\right)
+ \tilde \theta^{-1}\!\left(\eta^{\gamma}\right)
\leq \overline{c}\alpha
\]
for some constants $\underline{c},\overline{c}>0$,  
then we obtain the convergence rates
\begin{align*}
\breg{\ualdel} &=\mathcal O \left(\psi\left(
\theta^{-1}\!\left(\varepsilon\right)
+\tilde \theta^{-1}\left(\eta^{\gamma}\right)\right)\right) 
&&\qquad\text{as}\qquad \max\left\{\varepsilon, \eta\right\} \searrow 0, \\[0.1cm]
\norm{F\left(\ualdel\right) - \gdag}{L^1\left(\manifold\right)} 
&= \mathcal O \left(\varepsilon + 
\theta\!\paren{\tilde\theta^{-1}\paren{\eta^{\gamma}}}\right)
&&\qquad\text{as}\qquad \max\left\{\varepsilon, \eta\right\} \searrow 0.
\end{align*}
\emph{Case 2: $\varphi(t)=ct$ for some $c>0$.} 
If $0<\alpha\leq \frac{1}{2\Cerr c}$, then
\begin{align*}
\breg{\ualdel} + \norm{F\left(\ualdel\right) - \gdag}{L^1\left(\manifold\right)} 
=\mathcal O \left(\varepsilon+\eta^{\gamma}\right)
\qquad\text{as}\qquad \max\left\{\varepsilon, \eta\right\} \searrow 0.
\end{align*}
\end{thm}

\begin{proof}
\emph{Case 1:} A simple argument shows that the 
concavity of $\varphi^{1+\delta}$ implies 
$\psi(Ct)\leq C^{1/\delta} \psi\left(t\right)$ for all $C\geq 1$ and $t>0$ 
(see the proof of \cite[Thm.~5.1]{wh12}). 
From \eqref{eq:error_decomp_final} and the choice of $\alpha$ we obtain that
\begin{align*}
\breg{\ualdel} &\leq \frac{2q'}{\beta} \frac{\varepsilon}{\alpha} + \frac{q'-1}{\beta} \frac{\eta^{\gamma}}{\alpha^{q'}} 
+ \frac{C_{\psi}}{\beta} \psi\left(\Cerr \alpha\right)  
\\[0.1cm] & 
\leq C\left(\frac{\varepsilon}{\theta^{-1}\!\left(\varepsilon\right)} + \frac{\eta^{\gamma}}{\big(\tilde\theta^{-1}\!\left(\eta^{\gamma}\right)\big)^{q'}} + \psi\left(\underline{c} \alpha\right)\right) \\[0.1cm] 
& = C \paren{\psi\big(\theta^{-1}(\varepsilon)\big)
+ \psi\big(\tilde\theta^{-1}(\eta^{\gamma})\big)
+ \psi\left(\underline{c} \alpha\right)}  
\\[0.1cm] &
\leq 3C\psi\left(\theta^{-1}\!\left(\varepsilon\right)
+ \tilde \theta^{-1}\left(\eta^{\gamma}\right)\right)
\end{align*}
with $C:=\max\big\{\overline{c}\frac{2q'}{\beta},\overline{c}^{q'}\frac{q'-1}{\beta},
\frac{C_{\psi}}{\beta}\max\{\Cerr/\underline{c},1\}^{1/\delta}\big\}$. 
Here the equality follows from the identities 
$\varepsilon = \theta^{-1}(\varepsilon)\psi(\theta^{-1}(\varepsilon))$ 
and 
$\eta^{\gamma} = (\tilde\theta^{-1}(\eta^{\gamma}))^{q'}
\psi(\tilde\theta^{-1}(\eta^{\gamma}))$. 
Similarly it follows from \eqref{eq:error_decomp_final_Y}
and the identity 
$\tilde\theta^{-1}(\eta^{\gamma})^{q'-1} \cdot
\theta\!\paren{\tilde\theta^{-1}(\eta^{\gamma})} 
= \tilde\theta\!\paren{\tilde\theta^{-1}(\eta^{\gamma})} 
=\eta^{\gamma}$ 
that
\begin{align*}
\norm{F(\ualdel) - \gdag}{L^1\left(\manifold\right)}
&\leq 4q' \varepsilon + 2(q'-1) \frac{\eta^{\gamma}}{\alpha^{q'-1}} 
+ 2C_{\psi}\Cerr\alpha \psi\left(2 \Cerr \alpha\right) \\
&\leq \tilde{C}
\paren{\varepsilon + \frac{\eta^{\gamma}}
{\tilde\theta^{-1}\!\left(\eta^{\gamma}\right)^{q'-1}} 
+ \underline{c}\alpha \psi\left(\underline{c} \alpha\right)}\\
&= \tilde{C}\paren{\varepsilon 
+ \theta\!\paren{\tilde\theta^{-1}(\eta^{\gamma})} 
+ \theta(\underline{c}\alpha)}\\
&\leq 2\tilde{C}\paren{\varepsilon 
+ \theta\!\paren{\tilde\theta^{-1}(\eta^{\gamma})}}
\end{align*}
with $\tilde{C}:=\max\braces{4q',2(q'-1)\underline{c}^{1-q'},
2C_{\psi}\Cerr/\underline{c}\max\{2\Cerr/\underline{c},1\}^{1/\delta}}$. 

\emph{Case 2:} This follows immediately from 
Theorem~\ref{thm:impulsive_noise_conv_2} and the fact 
the $\psi(t)=0$ for $0\leq t\leq 1/c$ (see Remark \ref{rem:exact_penalization}).
\end{proof}

Note that the case distinction in 
Theorem~\ref{thm:rate_epsilon_eta} is not exhaustive: There 
are concave, almost linear index functions $\varphi$, which do 
not belong to any of the two classes and would require a 
separate discussion.

If $\varphi$ is a power function, 
the right hand side of the error bound is given more explicitly as follows: 

\begin{cor}\label{cor:power}
Suppose the assumptions of Theorem~\ref{thm:impulsive_noise_conv_2} hold true, let $\varphi$ in \eqref{eq:vie} 
be given by $\varphi \left(t\right) = c \cdot t^{\kappa}$ with $c > 0$ and $\kappa \in \left(0,1\right)$, 
and let $q=q'=2$. Then for 
$\alpha \sim \varepsilon^{1-\kappa}+ 
\eta^{\frac{1-\kappa}{2-\kappa}\gamma}$ we obtain
\begin{subequations}\label{eqs:rates_hoelder}
\begin{align}
\label{eq:rate_hoelderX}
\breg{\ualdel} &=\mathcal O \left(\varepsilon^{\kappa}+ 
\eta^{\frac{\kappa}{2-\kappa}\gamma}\right)
&&\qquad\text{as}\qquad \max\left\{\varepsilon, \eta\right\} \searrow 0, \\[0.1cm]
\label{eq:rate_hoelderY}
\norm{F\left(\ualdel\right) - \gdag}{L^1\left(\manifold\right)} &= \mathcal O \left(\varepsilon 
+\eta^{\frac{\gamma}{2-\kappa}}\right) 
&&\qquad\text{as}\qquad \max\left\{\varepsilon, \eta\right\} \searrow 0.
\end{align}
\end{subequations}
\end{cor}
\begin{proof}
This follows from $\psi\left(t\right) = C \cdot t^{\frac{\kappa}{1-\kappa}}$, 
$\theta \left(\alpha\right) = C\cdot\alpha^{\frac{1}{1-\kappa}}$, and
$\tilde \theta \left(\alpha\right)= C\cdot\alpha^{\frac{2-\kappa}{1-\kappa}}$ with 
$C:=c^{\frac{1}{1-\kappa}}
\paren{\kappa^{\frac{\kappa}{1-\kappa}}
-\kappa^{\frac{1}{1-\kappa}}}$.
\end{proof}

\subsection{functional dependence of $\varepsilon$ and $\eta$}\label{sec:varepsilon}
As the choices of $\varepsilon \geq 0$ and $\eta \geq 0$ are not independent of  each other, 
let us study the function 
\begin{equation}\label{eq:epsi_eta}
\varepsilon_{\xi}\left(\eta\right):=\inf\left\{\norm{\xi}{L^1\left(\mi\right)} ~\big|~\mc\in \Borel(\manifold), |\mc|\leq \eta\right\}\,.
\end{equation}

\begin{prop}\label{prop:varepsilon}
If $\xi \in L^1 \left(\manifold\right)$, the function $\varepsilon_{\xi}$ has the following properties:
\begin{enumerate}
\item $\varepsilon_{\xi}$ is continuous, decreasing, and convex. 
\item $\varepsilon_{\xi}(0)=\norm{\xi}{L^1\left(\manifold\right)}$, $\varepsilon_{\xi}(\meas{\manifold})=0$, 
and $\varepsilon_{\xi}$ is affine linear on $[0,\meas{\manifold}]$ if and only if 
$\left|\xi\right|$ is constant.
\item If $\xi\in L^{\infty}(\manifold)$, then
$
\lim_{t\searrow 0}\frac{1}{t}(\varepsilon_{\xi}(t)-\varepsilon_{\xi}(0))= -\norm{\xi}{L^{\infty}}$\,.
\end{enumerate}
\end{prop}

\begin{proof}
(i) \emph{In the first part of the proof we show that there exists a familiy of Borel sets 
$\{\mc_{\eta}\in\Borel(\manifold):\eta\in [0,\meas{\manifold}]\}$ and a decreasing function 
$a:[0,\meas{\manifold}]\to [0,\esssup|\xi|]$, which is continuous from the right, i.e.\ 
$\lim_{\eta\searrow \eta_0}a(\eta) = a(\eta_0)$ for all 
$\eta_0\in [0,\meas{\manifold})$, such that }
\begin{align}\label{eq:constr_Peta}
\begin{aligned}
&|\xi|\geq a(\eta) && \mbox{a.e.\ on } \mc_\eta,\qquad 
&|\xi|\leq a(\eta) && \mbox{a.e.\ on } \mi_\eta,\\
&\mc_{\underline{\eta}}\subset \mc_{\overline{\eta}} &&\mbox{if } \underline{\eta}\leq \overline{\eta}, \qquad &\meas{\mc_\eta} = \eta\,.&& 
\end{aligned}
\end{align}
To construct $\{\mc_\eta\}$ and $a$, define 
\(
b(\lambda):= \meas{\braces{x\in\manifold:|\xi(x)|\geq \lambda}}
\) 
for $\lambda\in [0,\esssup|\xi|]$. Then $b$ is decreasing, and 
a straightforward application of Lebesgue's Dominated Convergence Theorem shows that  
\begin{subequations}
\begin{align}
\label{eq:b_cont_left}
&\lim_{\tilde{\lambda}\nearrow \lambda}b(\tilde{\lambda}) = b(\lambda),\\
&\lim_{\tilde{\lambda}\searrow \lambda}b(\lambda) - b(\tilde{\lambda})
=\meas{C_{\lambda}} \qquad\mbox{with }
C_{\lambda}:=\braces{x\in\manifold:|\xi(x)|=\lambda}
\end{align}
\end{subequations}
for all $\lambda\in  (0,\esssup|\xi|)$. 
We define $a(\eta):=\max\{\lambda\geq 0:b(\lambda)\geq \eta\}$ (see Figure \ref{fig:epsilon} for a sketch of $a$ and $b$). Since $b$ is decreasing 
and continuous from the left, $a$ is decreasing and continuous from the right. 
We have  $b(a(\eta))= \lim_{\tilde{\lambda}\nearrow a(\eta)}b(\tilde{\lambda})$, and if
$b(a(\eta))= \eta$, then $a(\eta-t)=a(\eta)$ for all $0\leq t\leq \meas{C_{a(\eta)}}$. 
If $b$ is continuous at $a(\eta)$, we set $\mc_{\eta}:=\{x\in\manifold:|\xi(x)|\geq a(\eta)\}$.
Otherwise, choose $\eta$ with $b(a(\eta))= \eta$, 
let $\tilde{\mc}_\eta:=\{x\in\manifold:|\xi(x)|>a(\eta)\}$ 
and define $\mc_{\eta-t}:= \tilde{\mc}_{\eta}\cup (C_{\eta}\cap B(0,r(\eta,t)))$ where
$r(\eta,t)>0$ is chosen such that $\meas{\mc_{\eta-t}}=\eta-t$ for all 
$0\leq t\leq \meas{C_{a(\eta)}}$. It is easy to see that \eqref{eq:constr_Peta} is satisfied. 

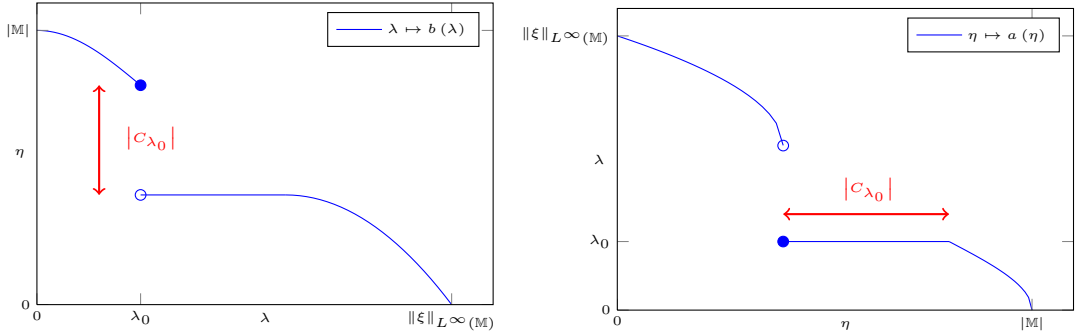
\begin{figure}[!htb]
\setlength\fheight{4cm} \setlength\fwidth{6cm}
\centering
\tiny
\begin{tikzpicture}
\begin{axis}[%
scale only axis,
xlabel = $\lambda$,
ylabel = $\eta$,
xlabel shift = -11pt,
ylabel shift = -12pt,
ylabel style = {rotate=270},
width=\fwidth,
height=\fheight,
xmin=0, xmax=1.1,
ymin=0, ymax=1.1,
xtick = {0,0.25,1},
xticklabels = {0,$\lambda_0$,$\norm{\xi}{L^\infty \left(\manifold\right)}$},
ytick = {0,1},
yticklabels = {0,$\meas{\manifold}$},
axis on top,
legend entries= {$\lambda \mapsto b\left(\lambda\right)$},
legend pos = north east]
\addplot [
color=blue,
solid,
domain = 0:0.25
]
{0.2*cos(deg(2*pi*x))+0.8};
\addplot [
color=blue,
mark = *
]
coordinates{(0.25,0.8)};
\addplot [
arrows = <->,
color=red,
thick,
solid
]
coordinates{(0.15,0.4)(0.15,0.8)};
\pgfplotsset{every axis/.append style={
extra description/.code={
\node at (0.25,0.55) {\textcolor{red}{$\meas{C_{\lambda_0}}$}};
}}}
\addplot [
color=blue,
mark = o
]
coordinates{(0.25,0.4)};
\addplot [
color=blue,
solid
]
coordinates{(0.25,0.4)(0.6,0.4)};
\addplot [
color=blue,
solid,
domain = 0.6:1
]
{-((x-0.6)/0.4*0.6325)^2+0.4};
\end{axis}
\end{tikzpicture}
\begin{tikzpicture}
\begin{axis}[%
scale only axis,
xlabel = $\eta$,
ylabel = $\lambda$,
xlabel shift = -8pt,
ylabel shift = -36pt,
ylabel style = {rotate=270},
width=\fwidth,
height=\fheight,
xmin=0, xmax=1.1,
ymin=0, ymax=1.1,
ytick = {0,0.25,1},
yticklabels = {0,$\lambda_0$,$\norm{\xi}{L^\infty \left(\manifold\right)}$},
xtick = {0,1},
xticklabels = {0,$\meas{\manifold}$},
axis on top,
legend entries= {$\eta \mapsto a\left(\eta\right)$},
legend pos = north east]]
\addplot [
color=blue,
solid,
domain = 0:0.4
]
{0.4/0.6325*sqrt(0.4-x)+0.6};
\addplot [
color=blue,
mark = o
]
coordinates{(0.4,0.6)};
\addplot [
arrows = <->,
color=red,
thick,
solid
]
coordinates{(0.4,0.35)(0.8,0.35)};
\pgfplotsset{every axis/.append style={
extra description/.code={
\node at (0.55,0.4) {\textcolor{red}{$\meas{C_{\lambda_0}}$}};
}}}
\addplot [
color=blue,
mark = *
]
coordinates{(0.4,0.25)};
\addplot [
color=blue,
solid
]
coordinates{(0.4,0.25)(0.8,0.25)};
\addplot [
color=blue,
solid,
domain = 0.8:1
]
{rad(acos(5*x-4))/2/pi};
\end{axis}
\end{tikzpicture}
\caption{Illustration of the functions $a$ and $b$ in the proof of Proposition~\ref{prop:varepsilon}.}
\label{fig:epsilon}
\end{figure}

(ii) \emph{relations between $\varepsilon_{\xi}$ and $\{\mc_{\eta}\}$, $a$:}
From \eqref{eq:constr_Peta} we obtain 
\begin{equation}\label{eq:char_varepsilon}
\varepsilon_{\xi}(\eta) = \norm{\xi}{L^1(\mi_{\eta})},\qquad \eta\in [0,\meas{\manifold}].
\end{equation}
For $0\leq \underline{\eta}\leq \overline{\eta}\leq\meas{\manifold}$ note that 
$\varepsilon_{\xi}(\underline{\eta})-\varepsilon_{\xi}(\overline{\eta})
= \norm{\xi}{L^1(\mc_{\overline{\eta}}\setminus \mc_{\underline{\eta}})}$. 
Due to \eqref{eq:constr_Peta} we obtain
\begin{equation}\label{eq:diffq_varepsilon}
a(\overline{\eta})\paren{\overline{\eta}-\underline{\eta}}
= a(\overline{\eta})\meas{\mc_{\overline{\eta}}\setminus \mc_{\underline{\eta}}}
\leq \varepsilon_{\xi}(\underline{\eta})-\varepsilon_{\xi}(\overline{\eta})
\leq  a(\underline{\eta})\meas{\mc_{\overline{\eta}}\setminus \mc_{\underline{\eta}}}
= a(\underline{\eta})\paren{\overline{\eta}-\underline{\eta}}\,.
\end{equation}
As $a$ is continuous from the right, it follows the $-a$ is the right-sided derivative of 
$\varepsilon_{\xi}$:
\begin{equation}\label{eq:deriv_varepsilon}
\lim_{\tau\searrow 0}\frac{\varepsilon_{\xi}(\eta+\tau)-\varepsilon_{\xi}(\eta)}{\tau} = -a(\eta)\,.
\end{equation}
(iii) \emph{Proof of the claims:}
\begin{enumerate}
\item Obviously, $\varepsilon_{\xi}$ is decreasing. Convexity follows from \eqref{eq:deriv_varepsilon}
and the fact that $a$ is decreasing. 
Continuity  follows from convexity and finiteness 
(see \cite[Cpt.~1, Cor.~2.3]{et76}).
\item The values of $\varepsilon_{\xi}$ at $0$ and $\meas{\manifold}$ follow immediately from 
the definition. It follows from \eqref{eq:deriv_varepsilon} that $\varepsilon_{\xi}$ is affine linear 
if and only if $a$ is constant, which in turn is equivalent that $|\xi|$ is constant. 
\item
This follows from \eqref{eq:deriv_varepsilon} and $a(0)=\norm{\xi}{L^\infty\left(\manifold\right)}$. 
\end{enumerate} 
\end{proof}

The more the graph of $\varepsilon_{\xi}$ looks like the letter 'L' or the faster $\varepsilon_{\xi}$ decays 
at $0$,  ``the more impulsive'' the noise $\xi$. Examples are shown in Figure \ref{fig:impulsive_noise}.

\begin{figure}[!htb]
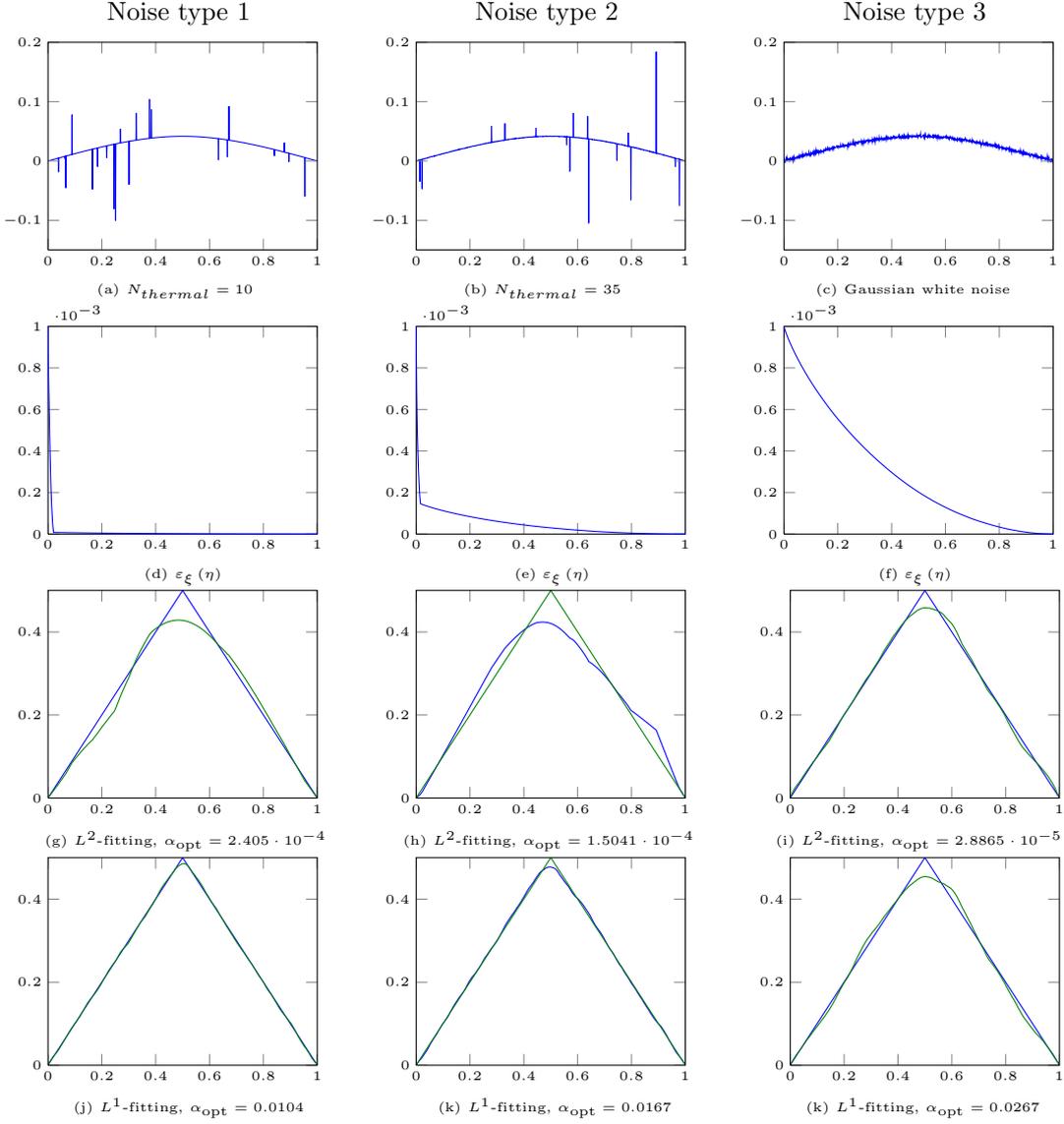

\setlength\fheight{2.8cm} \setlength\fwidth{3.6cm}
\centering
\tiny
\begin{tabular}{rrr}
\normalsize Noise type 1 \rule{1cm}{0pt} & \normalsize Noise type 2 \rule{1cm}{0pt} & \normalsize Noise type 3 \rule{1cm}{0pt} \\[0.1cm]
\input{impulsive_noise_1.tikz} & \input{impulsive_noise_2.tikz} & \input{impulsive_noise_3.tikz} \\[0.1cm]
(a)~$N_{thermal} = 10$ \rule{1cm}{0pt} & (b)~$N_{thermal} = 35$ \rule{1cm}{0pt} & (c)~Gaussian white noise\rule{.8cm}{0pt} \\[0.1cm]
\input{epsilon_eta_1.tikz} & \input{epsilon_eta_2.tikz} & \input{epsilon_eta_3.tikz} \\[0.1cm]
(d)~$\varepsilon_{\xi} \left(\eta\right)$\rule{1.5cm}{0pt} & (e)~$\varepsilon_{\xi} \left(\eta\right)$\rule{1.5cm}{0pt} & (f)~$\varepsilon_{\xi} \left(\eta\right)$ \rule{1.5cm}{0pt}\\[0.1cm]
\input{L2_rec_1.tikz} & \input{L2_rec_2.tikz} & \input{L2_rec_3.tikz}\\[0.1cm]
(g)~$L^2$-fitting, $\alpha_{\rm opt} = 2.405\cdot10^{-4}$\rule{.1cm}{0pt} & (h)~$L^2$-fitting, $\alpha_{\rm opt} =1.5041\cdot10^{-4} $\rule{.1cm}{0pt} &  (i)~$L^2$-fitting, $\alpha_{\rm opt} = 2.8865\cdot10^{-5}$\rule{.1cm}{0pt} \\[0.1cm]
\input{L1_rec_1.tikz} & \input{L1_rec_2.tikz} & \input{L1_rec_3.tikz}\\[0.1cm]
(j)~$L^1$-fitting, $\alpha_{\rm opt} = 0.0104$\rule{.4cm}{0pt}& (k)~$L^1$-fitting, $\alpha_{\rm opt} = 0.0167$\rule{.4cm}{0pt} &  (k)~$L^1$-fitting, $\alpha_{\rm opt} = 0.0267$\rule{.4cm}{0pt} \\[0.1cm]
\end{tabular}
\caption{For noise models 1 and 2 we choose $\xi = 0.001 \cdot \left(\xi_1 - \xi_2\right) / \norm{\xi_1 -\xi_2}{L^1 \left(\left[0,1\right]\right)}$ with $\xi_1, \xi_2$ simulated as proposed in \cite{ts10} for parameters $\beta = 100$, $A = 80$, $B = 7.5$, $f = 1000$ and $T =1$. The noise model 3 corresponds to similarly normed Gaussian white noise.}
\label{fig:impulsive_noise}
\end{figure}

\subsection{convergence rates in terms of an optimal $\eta$}

Substituting $\varepsilon$ by $\varepsilon_{\xi}(\eta)$ in 
\eqref{eq:rate_hoelderX} yields the error bound
\begin{equation}\label{eq:err_bound_infeta}
\breg{\ualdel} \leq C\inf_{0\leq \eta\leq |\manifold|}
\left[\varepsilon_{\xi}(\eta)^\kappa
+ \eta^{\frac{\kappa}{2-\kappa}\gamma}\right]
\end{equation}
for some constant $C>0$ and an optimal choice of $\alpha$.  Usually the function $\varepsilon_{\xi}$ will not
be known precisely since since $\xi$ in unknown, but in some situations 
an upper bound on $\varepsilon_{\xi}$ may be known. E.g.,\ if $\xi$ is 
the realization a random process with known distribution, we may be 
able to compute upper bounds on $\varepsilon_{\xi} $ with high probability. 
 
By Proposition~\ref{prop:varepsilon} the first term in 
the argument of the infimum is decreasing in $\eta$ 
whereas the second argument in increasing. 
By continuity and the monotonicity properties there exists 
some $\bar\eta$ such that 
\[
\varepsilon_{\xi} \left(\bar\eta\right) = \bar\eta^{\frac{\gamma}{2-\kappa}}.
\]
It is easy to see that 
$\varepsilon_{\xi}(\bar\eta)^{\kappa}
\leq \inf_{0\leq \eta\leq |\manifold|}
\left[\varepsilon_{\xi}(\eta)^\kappa+ \eta^{\frac{\kappa}{2-\kappa}\gamma}\right]
\leq 2\varepsilon_{\xi}(\bar\eta)^{\kappa}$ and hence 
\[
\breg{\ualdel} 
=2C \varepsilon_\xi\left(\bar\eta\right)^{\kappa}.  
\]
The standard error analysis would yield the convergence rate
\[
\breg{\ualdel} =\tilde{C}\norm{\xi}{L^1\left(\manifold\right)}^\kappa 
= \tilde{C} \varepsilon_\xi\left(0\right)^\kappa
\]
(see \eqref{eq:standard_err_bound}). Thus our analysis improves the 
known error bounds roughly by the factor 
\begin{equation}\label{eq:improvement}
\paren{\frac{\varepsilon_\xi\left(0\right)}{\varepsilon_\xi
\left(\bar\eta\right)}}^{\kappa}.
\end{equation}
Recall that for impulsive noise the graph of $\varepsilon_{\xi}$ is L-shaped, and thus $\bar\eta$ will be close to the corner 
of the L in such a case. 
Therefore the factor \eqref{eq:improvement} will 
be the larger the larger the impulsiveness of 
the noise. This is a heuristic argument that the improvement 
factor in \eqref{eq:improvement} may become arbitrarily large. Note that the convergence
rate of the residuals $\norm{F\left(\ualdel\right) - \gdag}{L^1\left(\manifold\right)}$ can
similarly be calculated as $\varepsilon_\xi\left(0\right)$ and $\varepsilon_\xi
\left(\bar\eta\right)$ respectively, so the corresponding impovement factor is just the $\kappa$-th root of \eqref{eq:improvement}.

The following example shows in fact an arbirary large improvement:

\begin{ex}[purely impulsive noise]\label{ex:pureIN}
Now let us investigate the case of purely impulsive noise close to 
Example~\ref{ex:delta_peaks}. By purely impulsive noise we mean 
a noise vector $\xi\in L^1\left(\manifold\right)$ that consists only of 
``impulses'', or more
precisely a noise vector $\xi$ for which we can also choose $\varepsilon=0$ 
and $\eta$ small in \eqref{eq:noise_model_1} (but not $\eta=0$ as
in Example~\ref{ex:delta_peaks}). More precisely, we consider a noise 
vector $\xi$ is such that
\begin{equation}\label{eq:xi_extreme}
\exists~\mc \in \Borel\left(\manifold\right): \qquad \meas{\mc} = \eta_0, 
\qquad \xi_{|_{\mc}} = \frac{s(\eta_0)}{\eta_0}, \qquad \xi_{|_{\mi}} = 0
\end{equation}
with some $\eta_0 > 0$ and some scaling factor $s(\eta_0)>0$ which may 
be chosen arbitrarily. Then one readily computes
\[
\varepsilon_{\xi} \left(\eta\right) = \begin{cases} s(\eta_0) \left(1-\frac{\eta}{\eta_0}\right) & \text{if } \eta < \eta_0, \\[0.1cm] 0 &\text{else.} \end{cases}
\]
Setting $\eta=\eta_0$ and $\eta=0$ in \eqref{eq:err_bound_infeta} we find
\[
\breg{\ualdel} \leq C \min\braces{\eta_0^{\frac{\kappa}{2-\kappa}\gamma},s(\eta_0)^\kappa}
\]
for an optimal $\alpha$. (Setting $\lambda:=\eta/\eta_0$ we see that other choices of $\eta\in [0,\eta_0]$ can 
improve the constant $C$ at most by the factor 
$\min_{0\leq\lambda\leq 1}\max\{(1-\lambda)^\kappa,\lambda^{\frac{\kappa}{2-\kappa}\gamma}\}$.) 
For a comparison let us calculate the corresponding noise levels 
$\norm{\xi}{L^1\left(\manifold\right)}$ and 
$\norm{\xi}{L^2 \left(\manifold\right)}$ in this situation:
\begin{subequations}
\begin{align}
\left\Vert \xi\right\Vert_{L^2 \left(\manifold\right)}^2 
= \left\Vert \xi\right\Vert_{L^2 \left(\mi\right)}^2
+  \left\Vert \xi\right\Vert_{L^2 \left(\mc\right)}^2 
= \meas{\mc} \frac{s(\eta_0)^2}{\eta_0^2}
&= \frac{s(\eta_0)^2}{\eta_0}, \\[0.1cm]
\left\Vert \xi\right\Vert_{L^1 \left(\manifold\right)} 
=\left\Vert \xi\right\Vert_{L^1 \left(\mi\right)}+  \left\Vert \xi\right\Vert_{L^1 \left(\mc\right)} 
=\meas{\mc} \frac{s(\eta_0)}{\eta_0} &=s(\eta_0).
\end{align}
\end{subequations}
We assume that $\udag$ satisfies the variational source condition \eqref{eq:vie} with index function 
$\varphi(t)=c\cdot t^{\kappa}$ for $\Y= L^1(\manifold)$ and $r=1$ or $\varphi(t) = \tilde{c}\cdot t^{\tilde{\kappa}}$ 
for $\Y=L^2(\manifold)$ and $r=2$. Moreover, we suppose that Assumptions \ref{ass:wellposed} and \ref{ass:T} hold true. 
In Table \ref{tab:rates} we collect both the standard  error bounds in Theorem \ref{thm:conv} with 
$\err\equiv 2\|\xi\|_{L^2(\manifold)}^2$ and $\err\equiv 2\|\xi\|_{L^1(\manifold)}$, rsp., and our new error bounds.

\begin{table}[!htb]
\footnotesize
\centering
\begin{tabular}{l|c|cc}
\toprule[2pt]
& $L^2$--Tikhonov regularization & \multicolumn{2}{c}{$L^1$--Tikhonov regularization} \\[0.1cm]
 error analysis & standard& standard & new \\[0.1cm]
\midrule[1pt]
$\frac{1}{2}\err\left(F\left(\ualdel\right)\right)\leq$ 
& $\norm{\xi}{L^2}^2=\frac{s(\eta_0)^2}{\eta_0}$ 
& $\norm{\xi}{L^1}=s(\eta_0)$
& $\mathcal O \left(\frac{\eta_0^{2\gamma}}{\alpha} 
+ \eta_0^{\gamma}\alpha^{\frac{\kappa}{2-2\kappa}}\right)$ \\[0.1cm]
$\breg{\ualdel}=$ 
& $\mathcal O \left(\frac{s(\eta_0)^{2\tilde{\kappa}}}{\eta_0^{\tilde{\kappa}}}\right)$ 
& $\mathcal O \left(s(\eta_0)^\kappa\right)$ 
& $\mathcal O \left(\min\braces{\eta_0^{\frac{\kappa\gamma}{2-\kappa}},s(\eta_0)^\kappa}\right)$ \\[0.1cm]
$\|F(\ualdel)-\gdag)\|_{L^1}=$ 
& $\mathcal O \left(\frac{s(\eta_0)}{\sqrt{\eta_0}}\right)$ 
& $\mathcal O \left(s(\eta_0)\right)$ 
& $\mathcal O \left(\min\braces{\eta_0^{\frac{\gamma}{2-\kappa}},s(\eta_0)}\right)$\\[0.1cm]
\bottomrule[2pt]
\end{tabular}
\caption{\label{tab:rates} Comparison of error bounds 
with noise vector $\xi$ as in \eqref{eq:xi_extreme}, see Example \ref{ex:pureIN}. 
Our new error bounds for $L^1$--Tikhonov regularization (with only the first terms 
in the $\min$) depend only on the size $\eta_0$ of the corrupted area, 
but not on the arbitrary scaling factor $s(\eta_0)$ bounding the 
noise in the corrupted area.} 
\end{table}

The standard error bounds for $L^2$-- and $L^1$--Tikhonov regularization are not immediately
comparable since the index functions in the source conditions may be different. Nevertheless, 
the standard error analysis shows an improvement of $L^1$-- over $L^2$--Tikhonov regularization 
in the sense that $s(\eta_0)/\sqrt{\eta_0}$ may explode as $\eta\to 0$ whereas 
$s(\eta_0)$ tends to $0$.
To measure the improvement of our new analysis over the standard analysis we can use the factor
$s(\eta_0)^\kappa / \min\braces{\eta_0^{\frac{\kappa\gamma}{2-\kappa}},s(\eta_0)^\kappa}$ which is an
analog to \eqref{eq:improvement}. Note that $\bar \eta$ and hence the value of \eqref{eq:improvement}
cannot be calculated in general. If we consider e.g. the case $s(\eta_0) = 1$ then the aformentioned
factor is given by $1 / \eta_0^{\frac{\kappa\gamma}{2-\kappa}} \to \infty$ as $\eta_0 \searrow 0$. 
This shows again that our new error bounds may improve the standard error bounds by an arbitrarily large factor. 
\end{ex}

\section{Numerical simulations}\label{sec:num}

In this section we compare the error bounds in Theorem~\ref{thm:impulsive_noise_conv_2} with  errors in numerical simulations. 
As an example we consider $\manifold =\left[0,1\right]$ and the linear integral operator 
$T : L^2 \left(\manifold\right) \to L^2 \left(\manifold\right)$ defined by
\begin{equation}\label{eq:int_op}
\left(Tf\right) \left(x\right) = \int\limits_0^1 k\left(x,y\right) f\left(y\right) \,\mathrm d y, \qquad x \in \manifold
\end{equation}
with kernel $k\left(x,y\right) = \min\left\{x\cdot\left(1-y\right), y \cdot \left(1-x\right)\right\}, x,y \in \manifold$. 
It is easy to see that $\left(Tf\right)'' = -f$ for all 
$f \in L^2 \left(\manifold\right)$. Moreover, $T$ satisfies 
Assumption~\ref{ass:T} with $k = 2, p = 2$ and $q = 2$, so 
$\gamma = 2k/d+2(p-1)/p =5$.
We discretized $T$ by choosing equidistant points 
$x_1 = \frac{1}{2n}, x_2 = \frac{3}{2n}, \dots, x_n = \frac{2n-1}{2n}$ and using the composite midpoint rule
\[
\left(Tf\right)\left(x\right) = \int\limits_0^1 k\left(x,y\right) f\left(y\right) \,\mathrm d y \approx \frac{1}{n} \sum\limits_{i=1}^n k\left(x, x_i\right) f\left(x_i\right)
\]
on the grid points $x = x_j$, $1 \leq j \leq n$. To avoid an inverse crime, the exact data $\gdag$ has always been calculated analytically.

For the implementation of the Tikhonov regularization \eqref{eq:tik_gen} with $L^1$ data fidelity term and  penalty 
$\R\left(\sol\right) = \frac12\norm{\sol}{L^2 \left(\manifold\right)}^2$ we use Fenchel duality as proposed in \cite{cjk10b}. 
Some calculations show that the Fenchel conjugates 
of $G\left(g\right) := \norm{g-\gobs}{L^1 \left(\manifold\right)}$ and 
$\R\left(\sol\right) = \frac12 \norm{\sol}{L^2 \left(\manifold\right)}^2$ 
are given by
\begin{align*}
G^* \left(p\right)&= \begin{cases} \left\langle p, \gobs\right\rangle & \text{if } \norm{p}{L^\infty \left(\manifold\right)} \leq 1, \\[0.1cm]
\infty & \text{else} \end{cases} \\[0.1cm]
\R^* \left(q\right) &=\frac12\norm{q}{L^2 \left(\manifold\right)}^2. 
\end{align*}
(see e.g.\ \cite[Cpt.~1, Def.~4.1]{et76})
Thus the dual problem (see e.g.\ \cite[Cpt.~3]{et76}) is in this case given by
\begin{equation}\label{eq:l1_tik_dual}
\paldel \in \argmax\limits_{\norm{p}{L^\infty \left(\manifold\right)} \leq \frac1\alpha} \left[-\frac12 \left\Vert T^*p \right\Vert_{L^2 \left(\manifold\right)}^2+\left\langle p, \gobs\right\rangle \right].
\end{equation}
The discretized version of this problem was solved by Matlab's \verb+quadprog+ routine. Finally, we calculated $\ualdel$ using the extremal relation $\ualdel = T^* \paldel$ (see e.g.\ \cite[Cpt.~3, Prop. 2.4]{et76} and note that 
$\partial \R \left(\sol\right) = \{\sol\}$).

We compared this to standard $L^2$--Tikhonov regularization
\begin{equation}\label{eq:tik_L2}
\ualdel \in \argmin\limits_{\sol \in L^2 \left(\manifold\right)} \left[\frac{1}{2\alpha}\norm{T\sol - \gobs}{L^2\left(\manifold\right)}^2 + \frac12\norm{\sol}{L^2 \left(\manifold\right)}^2\right]
= \{ \left(T^*T + \alpha I\right)^{-1} T^* \gobs\}
\end{equation}
for noise of different degrees of impulsiveness (see Figure \ref{fig:impulsive_noise}). The $L^1$ reconstructions 
are significantly more accurate than $L^2$ reconstructions for impulsive noise vectors whereas for white noise 
the $L^2$ reconstruction is slightly more accurate. 

Figure \ref{fig:conv_rates} shows rates of convergence for $L^1$ data fitting with two specific choices of $\udag$ having different degrees of smoothness. The degree of smoothness of $\udag$ in terms of the operator $T$ is shown by means of 
the index function $\varphi$ in \eqref{eq:vie}, which has been estimated 
by evaluating the approximation error 
$\left(-\varphi\right)^* \left(-\frac{1}{\alpha}\right)$ 
in \eqref{eq:error_decomp_breg} for many values of $\alpha$ and a numerical 
evaluation of the Fenchel transform. 

For the computations we generated impulsive noise close to \eqref{eq:xi_extreme} with $s\left(\eta_0\right) = 1$ (we also performed experiments with larger values of $s\left(\eta_0\right)$ which yielded almost identical results). To generate the noise vectors we randomly selected $\lceil\eta_0 \cdot n\rceil$ grid points which then form the set $\mc$, and afterwards set $\xi_{|_\mc}  = \pm 1/\eta_0$ with probability $\frac12$ respectively for each $x_i \in \mc$ in the manner of salt-and-pepper noise. 

For prechosen noise parameters $\eta_0^i = (4/5)^{i}$, $i = 1,\dots$ we performed $10$ experiments for each  
parameter value. The regularization parameter $\alpha$ was chosen optimally by trial and 
error for each experiment. In the plots the mean errors are plotted against 
$\eta$. Within the error tolerances the experimental rates of convergence 
agree well with the rates of convergence predicted by our analysis.

\begin{figure}[!htb]
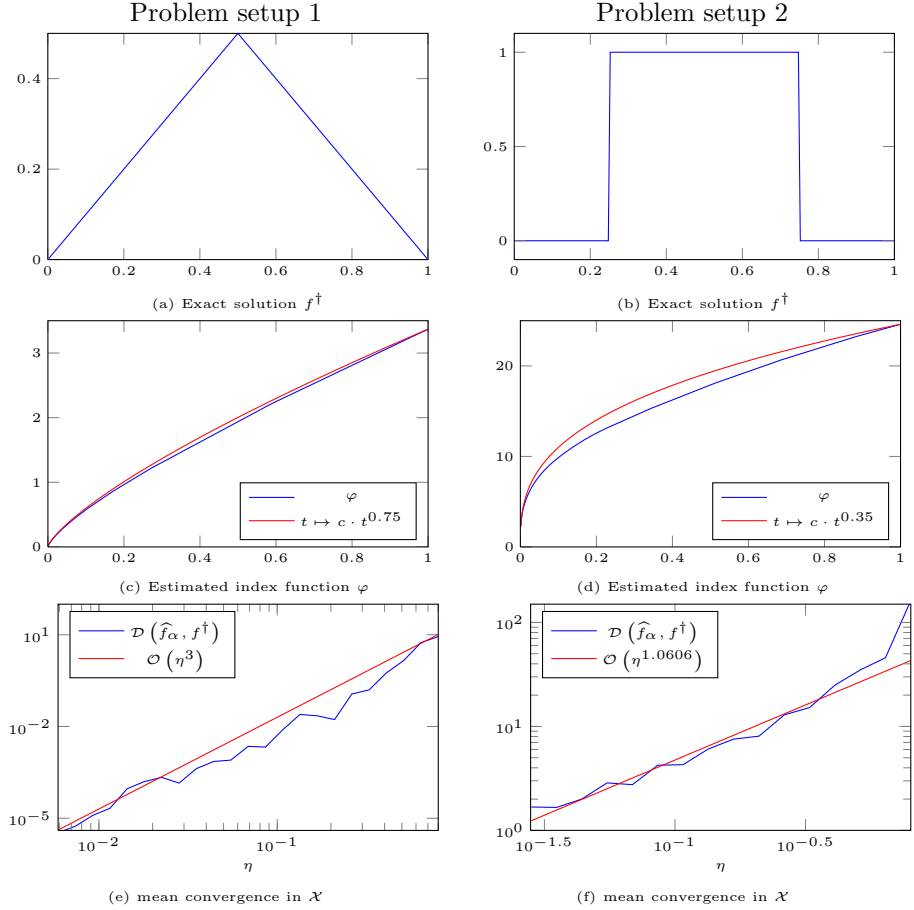

\setlength\fheight{3cm} \setlength\fwidth{5cm} 
\centering
\tiny
\begin{tabular}{rr}
\normalsize Problem setup 1 \rule{1.5cm}{0pt} & \normalsize Problem setup 2 \rule{1.5cm}{0pt} \\[0.1cm]
%
%
\begin{tikzpicture}

\begin{axis}[%
scale only axis,
width=\fwidth,
height=\fheight,
xmin=0, xmax=1,
ymin=0, ymax=0.5,
axis on top]
\addplot [
color=blue,
solid
]
coordinates{
 (0.00248756,0.00248756)(0.00746269,0.00746269)(0.0124378,0.0124378)(0.0174129,0.0174129)(0.0223881,0.0223881)(0.0273632,0.0273632)(0.0323383,0.0323383)(0.0373134,0.0373134)(0.0422886,0.0422886)(0.0472637,0.0472637)(0.0522388,0.0522388)(0.0572139,0.0572139)(0.0621891,0.0621891)(0.0671642,0.0671642)(0.0721393,0.0721393)(0.0771144,0.0771144)(0.0820896,0.0820896)(0.0870647,0.0870647)(0.0920398,0.0920398)(0.0970149,0.0970149)(0.10199,0.10199)(0.106965,0.106965)(0.11194,0.11194)(0.116915,0.116915)(0.121891,0.121891)(0.126866,0.126866)(0.131841,0.131841)(0.136816,0.136816)(0.141791,0.141791)(0.146766,0.146766)(0.151741,0.151741)(0.156716,0.156716)(0.161692,0.161692)(0.166667,0.166667)(0.171642,0.171642)(0.176617,0.176617)(0.181592,0.181592)(0.186567,0.186567)(0.191542,0.191542)(0.196517,0.196517)(0.201493,0.201493)(0.206468,0.206468)(0.211443,0.211443)(0.216418,0.216418)(0.221393,0.221393)(0.226368,0.226368)(0.231343,0.231343)(0.236318,0.236318)(0.241294,0.241294)(0.246269,0.246269)(0.251244,0.251244)(0.256219,0.
256219)(0.261194,0.261194)(0.266169,0.266169)(0.271144,0.271144)(0.276119,0.276119)(0.281095,0.281095)(0.28607,0.28607)(0.291045,0.291045)(0.29602,0.29602)(0.300995,0.300995)(0.30597,0.30597)(0.310945,0.310945)(0.31592,0.31592)(0.320896,0.320896)(0.325871,0.325871)(0.330846,0.330846)(0.335821,0.335821)(0.340796,0.340796)(0.345771,0.345771)(0.350746,0.350746)(0.355721,0.355721)(0.360697,0.360697)(0.365672,0.365672)(0.370647,0.370647)(0.375622,0.375622)(0.380597,0.380597)(0.385572,0.385572)(0.390547,0.390547)(0.395522,0.395522)(0.400498,0.400498)(0.405473,0.405473)(0.410448,0.410448)(0.415423,0.415423)(0.420398,0.420398)(0.425373,0.425373)(0.430348,0.430348)(0.435323,0.435323)(0.440299,0.440299)(0.445274,0.445274)(0.450249,0.450249)(0.455224,0.455224)(0.460199,0.460199)(0.465174,0.465174)(0.470149,0.470149)(0.475124,0.475124)(0.4801,0.4801)(0.485075,0.485075)(0.49005,0.49005)(0.495025,0.495025)(0.5,0.5)(0.504975,0.495025)(0.50995,0.49005)(0.514925,0.485075)(0.5199,0.4801)(0.524876,0.475124)(0.529851,0.470149)(
0.534826,0.465174)(0.539801,0.460199)(0.544776,0.455224)(0.549751,0.450249)(0.554726,0.445274)(0.559701,0.440299)(0.564677,0.435323)(0.569652,0.430348)(0.574627,0.425373)(0.579602,0.420398)(0.584577,0.415423)(0.589552,0.410448)(0.594527,0.405473)(0.599502,0.400498)(0.604478,0.395522)(0.609453,0.390547)(0.614428,0.385572)(0.619403,0.380597)(0.624378,0.375622)(0.629353,0.370647)(0.634328,0.365672)(0.639303,0.360697)(0.644279,0.355721)(0.649254,0.350746)(0.654229,0.345771)(0.659204,0.340796)(0.664179,0.335821)(0.669154,0.330846)(0.674129,0.325871)(0.679104,0.320896)(0.68408,0.31592)(0.689055,0.310945)(0.69403,0.30597)(0.699005,0.300995)(0.70398,0.29602)(0.708955,0.291045)(0.71393,0.28607)(0.718905,0.281095)(0.723881,0.276119)(0.728856,0.271144)(0.733831,0.266169)(0.738806,0.261194)(0.743781,0.256219)(0.748756,0.251244)(0.753731,0.246269)(0.758706,0.241294)(0.763682,0.236318)(0.768657,0.231343)(0.773632,0.226368)(0.778607,0.221393)(0.783582,0.216418)(0.788557,0.211443)(0.793532,0.206468)(0.798507,0.201493)(0.
803483,0.196517)(0.808458,0.191542)(0.813433,0.186567)(0.818408,0.181592)(0.823383,0.176617)(0.828358,0.171642)(0.833333,0.166667)(0.838308,0.161692)(0.843284,0.156716)(0.848259,0.151741)(0.853234,0.146766)(0.858209,0.141791)(0.863184,0.136816)(0.868159,0.131841)(0.873134,0.126866)(0.878109,0.121891)(0.883085,0.116915)(0.88806,0.11194)(0.893035,0.106965)(0.89801,0.10199)(0.902985,0.0970149)(0.90796,0.0920398)(0.912935,0.0870647)(0.91791,0.0820896)(0.922886,0.0771144)(0.927861,0.0721393)(0.932836,0.0671642)(0.937811,0.0621891)(0.942786,0.0572139)(0.947761,0.0522388)(0.952736,0.0472637)(0.957711,0.0422886)(0.962687,0.0373134)(0.967662,0.0323383)(0.972637,0.0273632)(0.977612,0.0223881)(0.982587,0.0174129)(0.987562,0.0124378)(0.992537,0.00746269)(0.997512,0.00248756) 
};

\end{axis}
\end{tikzpicture} & 
%
%
\begin{tikzpicture}

\begin{axis}[%
scale only axis,
width=\fwidth,
height=\fheight,
xmin=0, xmax=1,
ymin=-.1, ymax=1.1,
axis on top]
\addplot [
color=blue,
solid,
forget plot
]
coordinates{
(0,0)(0.2475,0)(0.2525,1)(0.7475,1)(0.7525,0)(1,0)
};

\end{axis}
\end{tikzpicture} \\[0.1cm]
(a)~Exact solution $\udag$ \rule{1.5cm}{0pt} & (b)~Exact solution $\udag$ \rule{1.5cm}{0pt} \\[0.1cm]
\input{problem_1_varphi.tikz} & \input{problem_2_varphi.tikz}\\[0.1cm]
(c)~Estimated index function $\varphi$\rule{1cm}{0pt} & (d)~Estimated index function $\varphi$ \rule{1cm}{0pt}\\[0.1cm]
%
%
\begin{tikzpicture}

\definecolor{mycolor1}{rgb}{0,0.5,0}
\definecolor{mycolor2}{rgb}{0,0.75,0.75}

\begin{loglogaxis}[%
scale only axis,
width=\fwidth,
height=\fheight,
xlabel = $\eta$,
xmin=0.0059, xmax=0.8,
ymin=4e-06, ymax=100,
axis on top,
legend entries={noise $=\xi$, noise $=\abs{\xi}\xi$, $\mathcal O \left(\eta^{3}\right)$},
legend entries={$\breg{\ualdel}$, $\mathcal O \left(\eta^{3}\right)$},
legend pos = north west
]
\addplot [
color=blue,
solid
]
coordinates{
 (0.8,8.68484)(0.64,5.52476)(0.512,1.43617)(0.4096,0.567907)(0.32768,0.158085)(0.262144,0.115257)(0.209715,0.0168724)(0.167772,0.0223457)(0.134218,0.0248146)(0.107374,0.0077959)(0.0858993,0.00211662)(0.0687195,0.00224623)(0.0549756,0.000798422)(0.0439805,0.000716815)(0.0351844,0.000418735)(0.0281475,0.000140919)(0.022518,0.000216909)(0.0180144,0.000158991)(0.0144115,9.25096e-05)(0.0115292,2.1147e-05)(0.00922337,1.23594e-05)(0.0073787,5.50376e-06)(0.00590296,3.33861e-06) 
};


\addplot [
color=red,
solid
]
coordinates{
 (0.8,10.2838)(0.64,5.2653)(0.512,2.69584)(0.4096,1.38027)(0.32768,0.706697)(0.262144,0.361829)(0.209715,0.185256)(0.167772,0.0948513)(0.134218,0.0485638)(0.107374,0.0248647)(0.0858993,0.0127307)(0.0687195,0.00651813)(0.0549756,0.00333728)(0.0439805,0.00170869)(0.0351844,0.000874849)(0.0281475,0.000447922)(0.022518,0.000229336)(0.0180144,0.00011742)(0.0144115,6.01191e-05)(0.0115292,3.0781e-05)(0.00922337,1.57599e-05)(0.0073787,8.06905e-06)(0.00590296,4.13136e-06) 
};
\end{loglogaxis}
\end{tikzpicture} & 
%
%
\begin{tikzpicture}

\definecolor{mycolor1}{rgb}{0,0.5,0}
\definecolor{mycolor2}{rgb}{0,0.75,0.75}

\begin{loglogaxis}[%
scale only axis,
width=\fwidth,
height=\fheight,
xlabel = $\eta$,
xmin=0.028, xmax=0.8,
ymin=1, ymax=150,
axis on top,
legend entries={$\breg{\ualdel}$, $\mathcal O \left(\eta^{1.0606}\right)$},
legend pos = north west
]
\addplot [
color=blue,
solid
]
coordinates{
 (0.8,162.363)(0.64,45.762)(0.512,34.9147)(0.4096,24.9176)(0.32768,15.1829)(0.262144,12.9139)(0.209715,8.05828)(0.167772,7.55617)(0.134218,6.06841)(0.107374,4.28584)(0.0858993,4.23294)(0.0687195,2.75237)(0.0549756,2.86458)(0.0439805,2.00301)(0.0351844,1.66291)(0.0281475,1.67943) 
};


\addplot [
color=red,
solid
]
coordinates{
 (0.8,43.0918)(0.64,34.0104)(0.512,26.8428)(0.4096,21.1858)(0.32768,16.721)(0.262144,13.1971)(0.209715,10.4158)(0.167772,8.22074)(0.134218,6.48825)(0.107374,5.12088)(0.0858993,4.04167)(0.0687195,3.18991)(0.0549756,2.51764)(0.0439805,1.98706)(0.0351844,1.56829)(0.0281475,1.23778) 
};
\end{loglogaxis}
\end{tikzpicture}\\[0.1cm]
(e)~mean convergence in $\X$\rule{1.6cm}{0pt} & (f)~mean convergence in $\X$\rule{1.6cm}{0pt}
\end{tabular}
\caption{Verification of the error estimates for the operator \eqref{eq:int_op} using the functions $\udag$ displayed in 
panels (a) and (b). Panels (c) and (d) show a numerical computation of the index functions $\varphi$ corresponding to 
these solutions (see text) and best fits of these functions of the form $\varphi(t)\approx ct^\kappa$. In Panels 
(e) and (f) we plot the error bounds from Theorem~\ref{thm:impulsive_noise_conv_2} against the experimental errors. 
(Here the multiplicative constant, which is not explicit in Theorem~\ref{thm:impulsive_noise_conv_2} was fitted to 
the experimental data.)
}
\label{fig:conv_rates}
\end{figure}

\section{Conclusions}
We have developed an error analysis of generalized Tikhonov regularization 
with $L^1$ fidelity term applied to inverse problems with impulsive 
noise. Our analysis is based on the deterministic, continuous noise 
model \eqref{eq:noise_model_1}. Numerical experiments suggest that 
the new error bounds are sharp (or at least almost sharp) up to constants 
whereas previous error bounds are far too pessimistic. 

Our analysis raises several questions for future research including 
lower bounds on the rate of convergence (which to our knowledge is 
an open question in all of the recent regularization theory in Banach spaces) 
and extensions of the results 
to stochastic noise models, more general data fidelity terms, and 
infinitely smoothing operators. 

\section*{Acknowledgement}

Financial support by the German Research Foundation DFG through 
SFB 755 and the Research Training Group 1023 is gratefully 
acknowledged. 

\small
\bibliography{impulsive_noise}{}

\begin{thebibliography}{10}

\bibitem{afhl12}
A.~Aravkin, M.~P. Friedlander, F.~J. Herrmann, and T.~van Leeuwen.
\newblock Robust inversion, dimensionality reduction, and randomized sampling.
\newblock {\em Math. Program.}, 134(1, Ser. B):101--125, 2012.

\bibitem{bl91}
J.~M. Borwein and A.~S. Lewis.
\newblock Convergence of best entropy estimates.
\newblock {\em SIAM J. Optimization}, 1:191--205, 1991.

\bibitem{bh10}
R.~I. Bot and B.~Hofmann.
\newblock An extension of the variational inequality approach for nonlinear
  ill-posed problems.
\newblock {\em J. Integral Equations Appl.}, 22(3):369--392, 2010.

\bibitem{bo04}
M.~Burger and S.~Osher.
\newblock Convergence rates of convex variational regularization.
\newblock {\em Inverse Probl.}, 20(5):1411--1422, 2004.

\bibitem{chn04}
R.~H. Chan, C.~Hu, and M.~Nikolova.
\newblock An iterative procedure for removing random-valued impulse noise.
\newblock {\em IEEE Signal Proc. Let.}, 11(12):921--924, 2004.

\bibitem{cj12}
C.~Clason and B.~Jin.
\newblock A semismooth {N}ewton method for nonlinear parameter identification
  problems with impulsive noise.
\newblock {\em SIAM J. Imaging Sci.}, 5(2):505--536, 2012.

\bibitem{cjk10b}
C.~Clason, B.~Jin, and K.~Kunisch.
\newblock A semismooth {N}ewton method for {$L^1$} data fitting with automatic
  choice of regularization parameters and noise calibration.
\newblock {\em SIAM J. Imaging Sci.}, 3(2):199--231, 2010.

\bibitem{e93}
P.~P.~B. Eggermont.
\newblock Maximum entropy regularization for {F}redholm integral equations of
  the first kind.
\newblock {\em SIAM J. Math. Anal.}, 24(6):1557--1576, 1993.

\bibitem{et76}
I.~Ekeland and R.~T{\'e}mam.
\newblock {\em Convex analysis and variational problems}.
\newblock Studies in mathematics and its applications. North Holland, 1976.

\bibitem{ehn96}
H.~Engl, M.~Hanke, and A.~Neubauer.
\newblock {\em Regularization of Inverse Problems}.
\newblock Springer, 1996.

\bibitem{f12}
J.~Flemming.
\newblock {\em Generalized {T}ikhonov regularization and modern convergence
  rate theory in {B}anach spaces}.
\newblock Shaker Verlag, Aachen, 2012.

\bibitem{fh10}
J.~Flemming and B.~Hofmann.
\newblock A new approach to source conditions in regularization with general
  residual term.
\newblock {\em Numer. Funct. Anal. Optimiz.}, 31:254--284, 2010.

\bibitem{fh11}
J.~Flemming and B.~Hofmann.
\newblock Convergence rates in constrained {T}ikhonov regularization:
  equivalence of projected source conditions and variational inequalities.
\newblock {\em Inverse Probl.}, 27(8):085001, 2011.

\bibitem{fhm11}
J.~Flemming, B.~Hofmann, and P.~Mathé.
\newblock Sharp converse results for the regularization error using distance
  functions.
\newblock {\em Inverse Probl.}, 27(2):025006, 2011.

\bibitem{g10b}
M.~Grasmair.
\newblock Generalized {B}regman distances and convergence rates for non-convex
  regularization methods.
\newblock {\em Inverse Probl.}, 26(11):115014, 2010.

\bibitem{hm12}
B.~Hofmann and P.~Mathé.
\newblock Parameter choice in {B}anach space regularization under variational
  inequalities.
\newblock {\em Inverse Problems}, 28(10):104006, 2012.

\bibitem{hy10}
B.~Hofmann and M.~Yamamoto.
\newblock On the interplay of source conditions and variational inequalities
  for nonlinear ill-posed problems.
\newblock {\em Appl. Anal.}, 89(11):1705--1727, 2010.

\bibitem{hw13}
T.~Hohage and F.~Werner.
\newblock Iteratively regularized {N}ewton-type methods for general data misfit
  functionals and applications to {P}oisson data.
\newblock {\em Numer. Math.}, 123(4):745--779, 2013.

\bibitem{j11}
B.~Jin.
\newblock A variational {B}ayesian method to inverse problems with impulsive
  noise.
\newblock {\em J. Comput. Phys.}, 231(2):423--435, 2012.

\bibitem{js12}
Q.~Jin and L.~Stals.
\newblock Nonstationary iterated {T}ikhonov regularization for ill-posed
  problems in {B}anach spaces.
\newblock {\em Inverse Probl.}, 28(10):104011, 2012.

\bibitem{kkm05}
T.~K{\"a}rkk{\"a}inen, K.~Kunisch, and K.~Majava.
\newblock Denoising of smooth images using {$L^1$}-fitting.
\newblock {\em Computing}, 74(4):353--376, 2005.

\bibitem{lsds11}
Y.-R. Li, L.~Shen, D.-Q. Dai, and B.~W. Suter.
\newblock Framelet algorithms for de-blurring images corrupted by impulse plus
  {G}aussian noise.
\newblock {\em IEEE Trans. Image Process.}, 20(7):1822--1837, 2011.

\bibitem{n02}
M.~Nikolova.
\newblock Minimizers of cost-functions involving nonsmooth data-fidelity terms.
  {A}pplication to the processing of outliers.
\newblock {\em SIAM J. Numer. Anal.}, 40(3):965--994 (electronic), 2002.

\bibitem{n04}
M.~Nikolova.
\newblock A variational approach to remove outliers and impulse noise.
\newblock {\em J. Math. Imaging Vision}, 20(1-2):99--120, 2004.
\newblock Special issue on mathematics and image analysis.

\bibitem{RR:92}
M.~Renardy and R.~C. Rogers.
\newblock {\em An Introduction to Partial Differential Equations}.
\newblock Springer, 1992.

\bibitem{r05}
E.~Resmerita.
\newblock Regularization of ill-posed problems in {B}anach spaces: convergence
  rates.
\newblock {\em Inverse Probl.}, 21(4):1303--1314, 2005.

\bibitem{s08}
O.~Scherzer, M.~Grasmair, H.~Grossauer, M.~Haltmeier, and F.~Lenzen.
\newblock {\em Variational Methods in Imaging}.
\newblock Applied Mathematical Sciences. Springer, 2008.

\bibitem{ts10}
P.~Tor\'{i}o and M.~G. S\'{a}nchez.
\newblock Generating impulsive noise.
\newblock {\em IEEE Antennas Propag.}, 52(4):168--173, August 2010.

\bibitem{wh12}
F.~Werner and T.~Hohage.
\newblock Convergence rates in expectation for {T}ikhonov-type regularization
  of {I}nverse {P}roblems with {P}oisson data.
\newblock {\em Inverse Probl.}, 28(10):104004, 2012.

\bibitem{w87}
J.~Wloka.
\newblock {\em Partial differential equations}.
\newblock Cambridge University Press, 1987.

\bibitem{xr91}
Z.~B. Xu and G.~F. Roach.
\newblock Characteristic inequalities of uniformly convex and uniformly smooth
  {B}anach spaces.
\newblock {\em J. Math. Anal. Appl.}, 157(1):189--210, 1991.

\bibitem{yzy09}
J.~Yang, Y.~Zhang, and W.~Yin.
\newblock An efficient {TVL}1 algorithm for deblurring multichannel images
  corrupted by impulsive noise.
\newblock {\em SIAM J. Sci. Comput.}, 31(4):2842--2865, 2009.

\bibitem{ygo07}
W.~Yin, D.~Goldfarb, and S.~Osher.
\newblock The total variation regularized {$L^1$} model for multiscale
  decomposition.
\newblock {\em Multiscale Model. Simul.}, 6(1):190--211 (electronic), 2007.

\end{thebibliography}
\bibliographystyle{plain}
\end{document}